\documentclass[10pt]{article}
\usepackage{flafter,amsmath,amssymb,latexsym,psfrag,graphicx,color,indentfirst}
\usepackage[margin=2.5cm]{geometry}
\usepackage{multirow}
\usepackage{subfig}
\usepackage{algorithm}
\usepackage{algorithmic}

\usepackage[colorlinks,
linktocpage=true,
linkcolor=black,
citecolor=black,
urlcolor=black,
anchorcolor=black
]{hyperref}
\usepackage{amsmath}
\usepackage{amsthm}
\usepackage{cases}
\usepackage{color}
\usepackage{ulem}
\usepackage{comment}
\usepackage{csquotes}
\usepackage{diagbox}

\usepackage[numbers,sort&compress]{natbib}

\DeclareGraphicsExtensions{.eps,.pdf,.jpg,.png}
\usepackage{epstopdf}
\allowdisplaybreaks

\newtheorem{theorem}{Theorem}[section]
\newtheorem{definition}[theorem]{Definition}
\newtheorem{lemma}[theorem]{Lemma}
\newtheorem{remark}[theorem]{Remark}

\newtheorem{proposition}[theorem]{Proposition}
\newtheorem{example}[theorem]{Example}
\newcommand{\red}[1]{\textcolor{red}{#1}}

\numberwithin{equation}{section}
\numberwithin{figure}{section}

\begin{document}
\setlength\arraycolsep{2pt}
\date{\today}
       
\title{Direct inversion scheme of time-domain fluorescence diffuse optical tomography
by asymptotic analysis of peak time}
\author{Shuli Chen,\; Junyong Eom\thanks{Corresponding author},\; Gen Nakamura,\; Goro  Nishimura}

\maketitle
\begin{abstract}
This paper proposes a direct inversion scheme for fluorescence diffuse optical tomography (FDOT) to reconstruct the location of a point target using the measured peak time of the temporal response functions. A sphere is defined for the target, with its radius determined by the peak time, indicating that the target lies on the sphere. By constructing a tetrahedron with edges determined by the radii, we identify the location of the target as the vertex of the tetrahedron.
Asymptotically, we derive the relationship between the radius of the sphere and the peak time. Several numerical tests are implemented to demonstrate the accuracy and performance of the asymptotic relationship and the inversion scheme.
\end{abstract}

\medskip
{\bf Keywords.} FDOT, Peak time, Asymptotic analysis, Reconstruction algorithm.

\medskip
{\bf MSC(2010):} 35R30, 35K20.

\vspace{10pt}
\noindent Addresses:

\noindent S. Chen:
School of Mathematics, Southeast University, Nanjing 210096, P. R. China. \\
\noindent 
E-mail: {\tt sli\_chen@126.com} \;
ORCID ID: {\tt 0009-0000-6563-5312}\\

\noindent J. Eom:  
Research Institute for Electronic Science, Hokkaido University, Sapporo 001-0020, Japan.\\
\noindent 
E-mail: {\tt eom@es.hokudai.ac.jp} \;
ORCID ID: {\tt 0000-0002-2749-3322}\\

\noindent G. Nakamura:  Department of Mathematics, Hokkaido University, Sapporo 060-0810,
Research Institute for Electronic Science, Hokkaido University, Sapporo 001-0020, Japan.\\
\noindent 
E-mail: {\tt nakamuragenn@gmail.com} \; 
ORCID ID: {\tt 0000-0002-7911-8612}\\

\noindent G. Nishimura:  
Research Institute for Electronic Science, Hokkaido University, Sapporo 001-0020, Japan.\\
\noindent 
E-mail: {\tt gnishi@es.hokudai.ac.jp} \;
ORCID ID: {\tt 0000-0003-4330-2626}\\

\noindent {\bf Acknowledgement}:
The first author was supported by the National Natural Science Foundation of China (No. 12071072) and China Scholarship Council (No. 202106090240). The second author was supported by the JSPS KAKENHI (Grant Number 23K19002).
The third author was supported by the JSPS KAKENHI (Grant Number JP22K03366).
The last author was supported by the JSPS/MEXT KAKENHI (Grant Numbers JP23H04127 and JP24K03305).

\newpage
\section{Introduction}
Fluorescence diffuse optical tomography (FDOT) is a fluorescence imaging technique for fluorophores embedded in highly scattering media. In particular, this technique is essential in biological or medical applications for tissues \textit{in vivo} to visualize specific diseases and biological activities using fluorescent probes from measurements outside of the tissue \cite{Ammari2020, Mycek2004, Vasilis2002}. Since the problems caused by the highly scattering media do not allow visualizing fluorophores in deep tissue using standard imaging techniques such as fluorescence camera, requiring a special imaging technique, FDOT. 
\begin{comment}
This imaging technique is also called fluorescence molecular tomography (FMT) and is categorized as a kind of diffuse optical tomography (DOT) using fluorescence. 
\end{comment}
In highly scattering media such as biological tissue, the light path is not straight anymore, and the repeating scattering makes the light propagation an energy dissipation process approximately described by a diffusion equation. 
The fluorescence image is blurred significantly, 
eventually, a reconstruction method based on the light propagation model is required to visualize and recover the quantitative information of three-dimensional fluorescence distribution from the measurements only on the boundary of the medium \cite{Jiang2011, Jiang2022}. 
\begin{comment}
The FDOT is categorized by the measurement data types \cite{Hebden1997, Hoshi2016}, such as the steady-state fluorescence intensity (CW method)\cite{Herve2011, Ducros2011}, the temporal response function of the fluorescence intensity (time-domain method)\cite{Lam2005, Nishimura2017}, the phase and demodulation of the fluorescence intensity (frequency-domain method)\cite{Ammari2014, Milstein2003} and a hybrid method of them \cite{Papadimitriou2020}. 
We choose the time-domain method because the temporal response function has direct information on the distribution of optical paths determined by the geometry, the position of the injection point of the excitation light, the distribution of the fluorophores, and the detection point of the fluorescence. 
\end{comment}

In this paper, we focus on a time-domain method of FDOT, measuring temporal response functions.
The temporal response function is the light intensity detected at certain times after an instantaneous injection of light, like a delta function. After the injection, light spreads in medium, and finally arrives the detection point after many different paths at different arrival times, corresponding to the traveling times of different path lengths \cite{Hebden1991}. In addition, there is a time delay in the temporal function due to the staying time in the excited state of the fluorophore determined by the fluorescence lifetime. 
\begin{comment}
Namely, the measured temporal response function is a discretized temporal response function, which is given by the solution of the light propagation model and reflects the optical path length distribution function associated with the target location. 
\end{comment}
Therefore, the temporal response function is determined by the path length distribution and the fluorescence lifetime.
Then, we use the peak time, defined as the time when the response function becomes maximum, for the inputs of the reconstruction of targets because it is least affected by noise and artifacts due to environmental contamination and should be robustly determined. The peak time corresponds to the most likely length of the light path , determined by the distances between the source, the target, and the detector.

Next, we start to formulate our inverse problem. 
\begin{comment}
We first consider the measurements of the fluorescence targets in tissue like the human chest, which is considerably larger than the measurable distances, resulting in the boundary being an approximately infinite plane. 
\end{comment}
Let $\Omega:={\mathbb R}_+^3$ and denote its boundary by $\partial\Omega$. Excitation and emission processes, $u_e$ and $U_m$, respectively, can be modeled as the following coupled diffusion equations \cite{Liu2022}:
\begin{equation}\label{ue_sys}
\left\{
\begin{array}{ll}
\left(v^{-1} \partial_t-D \Delta+\mu_a\right) u_e=0, & (x, t) \in \Omega \times(0, \infty), \\ 
u_e=0, & (x, t) \in \bar{\Omega} \times\{0\}, \\
\partial_\nu u_e+\beta u_e=\delta\left(x-x_s\right) \delta(t), & (x, t) \in \partial \Omega \times(0, \infty)
\end{array}
\right.
\end{equation}
and
\begin{equation}\label{um_sys}
\left\{
\begin{array}{ll}
\left(v^{-1} \partial_t-D \Delta+\mu_a\right) U_m=
\mu (f_\ell \ast u_e), & (x, t) \in \Omega \times(0, \infty), \\ 
U_m=0, & (x, t) \in \bar{\Omega} \times\{0\}, \\
\partial_\nu U_m+\beta U_m=0, & (x, t) \in \partial \Omega \times(0, \infty).
\end{array}
\right.
\end{equation}
Here, $\partial_\nu:=\nu \cdot \nabla$ is the exterior normal derivative, $v$ is the speed of light in the medium, $D$ is the diffusion constant, $\mu_a$ is the absorption coefficient, and $\beta=b/D>0$ is a positive constant with $b\in[0,1]$, coming from the Fresnel's reflection at the boundary due to the refractive index mismatch at the boundary. Also, $\delta(\cdot)$ denotes the delta function, and $x_s\in\partial\Omega $ is the position of the point source where the excitation light is injected. Further, the source term, which corresponds to the fluorescence emission from the fluorophores, for $U_m$ on the right-hand side of \eqref{um_sys} is given as  
\begin{equation}\label{source term}
\begin{split}
\mu (f_\ell \ast u_e) (x,t) 
&:= \mu(x) \int_0^t \ell^{-1}e^{-\frac{t-s}{\ell}} u_e(x, s; x_s)\,{\rm d}s,
\end{split}
\end{equation}
where  $\mu(x)>0$ is the absorption coefficient of a fluorophore, and $f_\ell\ast u_e$ is the convolution of $u_e$ and the fluorescence decay function, $f_\ell(t):=\ell^{-1}e^{-t/\ell},\,t\ge0$, with the fluorescence lifetime $\ell>0$. 
Note that if $\ell\rightarrow 0$, it is easy to see that $\mu(f_\ell\ast u_e)(x,t)=\mu(x)u_e(x,t;x_s)$. Hence, we extend the definition of \eqref{source term} to the case of the zero fluorescence lifetime $(\ell=0)$ by defining its right-hand side by $\mu(x)u_e(x,t;x_s)$.
\medskip

Then, our inverse problem is formulated as follows.

\noindent
\textbf{Inverse Problem}: Let $\mu(x)= c\delta(x-x_c)$,
where $x_c$ is the location of an unknown point target, and unknown $c>0$ is the absorption strength of committed fluorescence by the target at $x_c$.
Then, reconstruct $x_c$ from the measured peak times given as
$$t_{peak}^{(n)}:=t_{peak}(x_s^{(n)},\,x_d^{(n)})=\underset{t>0}{\arg \max}\,U_m(x_d^{(n)},\,t;\,x_s^{(n)}),\,n=1,\,2,\,\cdots,\,N,$$
where  $\left\{\{x_s^{(n)},\,x_d^{(n)}\}\right\}_{n=1}^N$ are $N$ sets of source and detector pairs (S-D pairs) consisting of source points $\{x_s^{(n)}\}_{n=1}^N$ and detector points $\{x_d^{(n)}\}_{n=1}^N$ located on $\partial\Omega$. 

\medskip

\begin{comment}
Now, we briefly review some related works to the mentioned inverse problem, where the peak time is studied or further used to solve the inverse problem. In \cite{Hall2004,Hall2010}, the authors considered the case of $\ell>0$ and $\beta=0$. They reconstructed the depth of the point target by numerically calculating the peak time without giving any of the formulas derived in a mathematically rigorous way. 
In \cite{Chen2023}, the authors of this paper derived the asymptotic behavior of the solution $U_m$ to \eqref{um_sys} and an approximate peak time equation (see \eqref{cubicpoly}) for the case $\ell=0, \; \beta>0$. Moreover, in \cite{Chen2025}, they derived an approximate peak time equation for the case $\ell >0, \; \beta>0$. The derived approximate peak time equations in \cite{Chen2023} and \cite{Chen2025} have shown a nice accuracy to peak time, which led them to propose a bisection algorithm to reconstruct the location of unknown point target. 
\end{comment}

Given S-D pair $(x_s, x_d)$ on the surface $\partial \Omega$ and an unknown target $x_c$ in $\Omega$, the peak time is determined by the distances between the source, the target, and the detector. We represent the distance by $\lambda>0$, where it satisfies
\begin{equation}\label{def_distance}
\lambda^2 = \frac{|x_d-x_c|^2+|x_s-x_c|^2}{2vD}.
\end{equation}
In Theorems \ref{thm_lzero} and \ref{thm_linfty}, we derive the asymptotic expansions for the approximate peak time with respect to the distance $\lambda \gg 1$, where the approximate peak time is defined as the approximation for the peak time.
\begin{comment}
\begin{equation}\label{expanforpeaktime}
t = \left\{
\begin{array}{ll}
k^{-\frac{1}{2}}\lambda - \frac{7}{4}k^{-1} + \frac{\sqrt{k}}{\sqrt{k}+\beta\sqrt{vD}}k^{-1} + \ell + O\left(\lambda^{-1}\right) &  \mbox{if} \quad \ell \ll 1, \vspace{5pt} \\ 
(k-\ell^{-1})^{-\frac{1}{2}} \lambda + (k-\ell^{-1})^{-\frac{3}{4}} \alpha_\lambda \lambda^\frac{1}{2} 
+ O\left( \lambda^{-\frac{1}{2}} \right) & 
\mbox{if} \quad \ell>\pi^{\frac{1}{2}}(k-\ell^{-1})^{-\frac{3}{4}}\lambda^{\frac{1}{2}},
\end{array}
\right.
\end{equation}
where
$$
k := v \mu_a >0, \quad 
\alpha_\lambda := \left(-\log\left[ (\pi^{\frac{1}{2}} \ell^{-1} (k-\ell^{-1})^{-\frac{3}{4}}) \lambda^\frac{1}{2} \right] \right)^\frac{1}{2} < \infty.
$$
\end{comment}
By inverting the role of the peak time $t>0$ and $\lambda>0$, we obtain the asymptotic expansions for the distance $\lambda>0$ 
\begin{equation}\label{expanforlambda}
\lambda(t) = \left\{
\begin{array}{ll}
k^\frac{1}{2}t + \frac{7}{4}k^{-\frac{1}{2}} - 
\frac{\sqrt{k}}{\sqrt{k} + \beta \sqrt{vD}}k^{-\frac{1}{2}}
-\ell k^\frac{1}{2} + O\left(t^{-1}\right) &  \mbox{if} \quad \ell \ll 1, \vspace{5pt} \\ 
(k-\ell^{-1})^\frac{1}{2}t - \tilde \alpha_t t^\frac{1}{2} + \frac{\tilde \alpha_t^2}{2} (k-\ell^{-1})^{-\frac{1}{2}} + O\left(t^{-\frac{1}{2}}\right)& 
\mbox{if} \quad \ell > \pi^{\frac{1}{2}} (k-\ell^{-1})^{-\frac{1}{2}} t^\frac{1}{2},
\end{array}
\right.
\end{equation}
as $t \gg 1$, where
\begin{equation}\label{def_alphatilde}
k := v \mu_a >0, \quad \tilde \alpha_t:=
 \left(-\log\left[ \pi^{\frac{1}{2}} \ell^{-1} (k-\ell^{-1})^{-\frac{1}{2}} t^\frac{1}{2} \right] \right)^\frac{1}{2}.
\end{equation}
The asymptotic expansion \eqref{expanforlambda} is derived in 
\eqref{lambdaexpansion} in detail.
From the equation \eqref{expanforlambda}, once we know the peak time, then we can determine the length $\lambda>0$ directly, which enables us to see how much the unknown target $x_c$ is far from given S-D pair.
Indeed, the equation \eqref{def_distance} says that the target $x_c$ lies on the sphere 
\begin{equation}\label{radii}
\begin{split}
\left|x_c - \frac{x_d + x_s}{2}\right|^2 
= r^2 \quad
\mbox{with the radius} \quad 
r = \sqrt{vD\lambda(t)^2
- \frac{|x_d - x_s|^2}{4}}. 
\end{split}
\end{equation}
The radius of sphere $r>0$ can be explicitly calculated by the peak time with a given S-D pair. 
Then, 
taking advantage of the information of the radius, 
we can find the unknown target $x_c$ as the vertex point of a tetrahedron with edges determined by the radii $r>0$ in \eqref{radii} (See Figure \ref{fig_tetrahedron} for the tetrahedron). Three S-D pairs on the plane $\partial \Omega$ make a base triangle of the tetrahedron, and the last vertex of the tetrahedron, which is the location of the target, can be explicitly calculated by the length of edges for the tetrahedron. 

As far as we know, there is almost no study on the explicit relationship between the peak time and the distance parameter which contains the information of the distance between the source, the target, and the detector. In \cite{Eom2023a}, confining to the case $\ell=0$, they show that the peak time can be represented by the solution for a certain quadratic polynomial whose coefficients are given by $\lambda>0$ and $k>0$ in \eqref{def_distance} and \eqref{def_alphatilde} but independent of $\beta>0$. 
In Theorems \ref{thm_lzero}, we can see the dependency of
$\beta>0$ on the peak time
\begin{equation*}
t =  k^{-\frac{1}{2}}\lambda - \frac{7}{4}k^{-1} + \frac{\sqrt{k}}{\sqrt{k}+\beta\sqrt{vD}}k^{-1} + \ell + O\left(\lambda^{-1}\right) \quad \mbox{as} \quad \lambda \gg 1. 
\end{equation*}
In \cite{Chen2025}, the authors of this paper derived the equation for the peak time
\begin{equation}\label{previousScheme}
 e^{-\frac{(\sqrt{k}t-\lambda)^2}{t}} = \pi^{\frac{1}{2}} \ell^{-1} \lambda^{-1} t^{\frac{3}{2}} 
\end{equation}
under the conditions $\ell \gg 1$ and $x_{c_3} \gg 1$.
Without these conditions, we generalize the equation \eqref{previousScheme} to 
\begin{equation}\label{ourScheme}
\begin{split}
 e^{-\frac{(\sqrt{k-\ell^{-1}}t-\lambda)^2}{t}} = \pi^{\frac{1}{2}} \ell^{-1} \lambda^{-1} t^{\frac{3}{2}} \left(\frac{x_{c_3}+\beta vDt }{x_{c_3}+\beta vD \lambda (k-\ell^{-1})^{-\frac{1}{2}}}\right)^2.
\end{split}
\end{equation}
Numerically we can see that the equation \eqref{ourScheme} is much more accurate than the previous one \eqref{previousScheme} (See Figure \ref{fig_tps_diffpara}). Based on \eqref{ourScheme}, we obtain the asymptotic expansion of the approximate peak time in Theorem \ref{thm_linfty}, where the approximate peak time is defined as the positive solution $t>0$ to \eqref{ourScheme}.  

The rest of this paper is organized as follows. In Section \ref{sec_asym}, we define the approximate peak time and derive the asymptotic expansions for the approximate peak time. In Section \ref{sec_numeric}, we numerically verify the accuracy of the approximate peak time to the peak time. In Section \ref{sec_inversion}, we derive the asymptotic expansion for the length parameter $\lambda>0$ and propose a direct inversion scheme. In Section \ref{sec_numericexp}, we show several numerical experiments to see the performance of the proposed inversion scheme.
%%%%%%%%%%%%%%%%%%%%%%%%%%%%%%%%%%%%%%%%%%%%%%%%%%%%%%%%%%%%%
%%%%%%%%%%%%%%%%%%%%%%%%%%%%%%%%%%%%%%%%%%%%%%%%%%%%%%%%%%%%%
\section{Asymptotic expansion of approximate peak time}\label{sec_asym}
In this section, we derive asymptotic expansion of approximate peak time in the case of fluorescence lifetime $\ell \ll 1$ and $\ell \gg 1$. In the subsection \ref{ellsmall}, we first consider the case $\ell \ll 1$.
%%%%%%%%%%%%%%%%%%%%%%%%%%%%%%%%%%%%%%%%%%%%%%%%%%%%%%%%%%%%%
\subsection{Case $\ell \ll 1$.}\label{ellsmall}
We first give the expression of zero fluorescence lifetime $\ell =0$ solution to \eqref{um_sys}
\begin{equation}\label{um-point2}
\begin{split}
u_m(x_d,t; x_s) = \frac{ce^{-v \mu_a t}}{16\pi^3 D^2v} \int_0^t \big((t-s)s\big)^{-\frac{3}{2}}  e^{-\frac{\|x_d-x_c\|^2}{4vD(t-s)}} e^{-\frac{\|x_s-x_c\|^2}{4vDs}} 
{\hat {\mathcal K}}(x_{c3};t-s) {\hat {\mathcal K}}(x_{c3};s) \,{\rm d}s,
\end{split}
\end{equation}
where
\begin{equation*}
\left\{
\begin{array}{ll}
{\hat {\mathcal K}}(x_{c_3};t) :=1-\beta\sqrt{\pi vDt} \,\exp{\left(\left(\frac{x_{c_3}+2\beta vDt}{\sqrt{4vDt}}\right)^2\right)}
\mathop{\mathrm{erfc}}
\left(\frac{x_{c_3}+2\beta vDt}{\sqrt{4vDt}}\right),\vspace{5pt} \\
{\rm erfc}(\xi)=\frac{2}{\sqrt{\pi}}\int_\xi^\infty e^{-s^2} \, {\rm d}s, \; \; \xi\in \mathbb{R}.
\end{array}
\right.
\end{equation*}
Here $\Vert\xi\Vert$ is the Euclidean distance of any three-dimensional vector $\xi$, and $x_{c_3}$ is the third component of the target $x_c$. 
See \cite[Section 2]{Eom2023b} for the derivation of \eqref{um-point2} in detail.

When $ 0 \neq \ell \ll 1$, the solution $U_m$ to \eqref{um_sys} satisfies 
\begin{equation*}
\begin{split}
U_m(t) &= \int_0^t \ell^{-1}e^{-\frac{t-s}{\ell}} u_m (s)  \,{\rm d}s  = u_m(t) - \int_0^t  e^{-\frac{t-s}{\ell}} \partial_s u_m (s)  \,{\rm d}s  \\
&= u_m(t) - \ell \partial_t u_m (t) + \ell^2 \partial^2_t u_m (t) + \cdots
\end{split}
\end{equation*}
and
\begin{equation*}
\begin{split}
\partial_t U_m(t) = \partial_t u_m(t) - \ell \partial^2_t u_m (t) + \ell^2 \partial^2_t u_m (t) + \cdots.
\end{split}
\end{equation*}
We approximate the critical point of $U_m$ by the positive solution $t>0$ to
\begin{equation}\label{ellsmalleq}
\begin{split}
\partial_t u_m(t) = \ell \partial^2_t u_m (t).
\end{split}
\end{equation}
In the next Lemma \ref{solbehaviortheorem1}, one can see the asymptotic behavior of ${u}_m$ when the target is far away from S-D pair (see \cite[Theorem 2.2 and Remark 2.3]{Chen2023}).
\begin{lemma}\label{solbehaviortheorem1}
Let $x_d, \; x_s\in\partial\Omega$, and assume that
\begin{equation}\label{uma_constr}
\begin{split}
\Big| |x_d-x_c|^2-|x_s-x_c|^2 \Big| \le C t \quad {\rm for\; some} \quad C>0.
\end{split}
\end{equation}
Define
\begin{equation}\label{defk}
k: = v \mu _a , \quad \lambda^2: = \frac{|x_d-x_c|^2+|x_s-x_c|^2}{2vD}.
\end{equation}
Then, $u_m$ satisfies
\begin{equation}\label{1st_asym_um}
\begin{split}
u_m(t) = u^a_m(t) +
O \left( u^a_m(t) \lambda^{-1}  \right),\;\; \lambda \gg 1,
\end{split}
\end{equation}
where
\begin{equation}\label{defIt}
\begin{split}
&u^a_m(t) = C_0 e^{-kt}  t^{-\frac{3}{2}} e^{-\frac{\lambda^2}{t}}
\left(\frac{x_{c_3}}{x_{c_3}+\beta v D t}\right)^2 \quad \mbox{with} \quad C_0 := \frac{c}{8\pi^\frac{5}{2} v^\frac{1}{2} D^\frac{3}{2} } \left( \frac{1}{|x_d-x_c|} + \frac{1}{|x_s-x_c|} \right).
\end{split}
\end{equation}
\end{lemma}
Taking advantage of the asymptotic profile $u_m^a$, we calculate the  critical point by
\begin{equation}\label{cubicpoly}
\begin{split}
\frac{d u_m^a}{ dt } &= P(t) u_m^a, \quad
P(t) := - k  - \frac{3}{2} t^{-1} + \lambda^2 t^{-2} 
- \frac{2\beta v D}{x_{c_3} + \beta v D t} = 0.
\end{split}
\end{equation}
The unique existence of the positive solution to $P(t)=0$ 
is guaranteed by the positivity of the physical constants in \eqref{defk} (see \cite[Theorem 3.1]{Chen2023}). In the following definition, we define approximate peak time by replacing $u_m$ as $u_m^a$ in the equation \eqref{ellsmalleq} and using \eqref{cubicpoly} in the case of  $\ell \ll 1$.
\begin{definition}\label{def_lzero}
We define approximate peak time $t^{p}_0$ by the positive solution to 
\begin{equation}\label{ellsmalldef}
\begin{split}
P(t) = \ell \left( P^{'}(t) + P(t)^2 \right)
\end{split}
\end{equation}
in the case of a small fluorescence lifetime.
\end{definition}
In the next Theorem \ref{thm_lzero}, we derive the asymptotic expansion of the approximate peak time $t^{p}_0$ when $\lambda \gg 1$.
\begin{theorem}\label{thm_lzero}
The approximate peak time for a small fluorescence lifetime satisfies
\begin{equation}\label{asympofpeak1}
\begin{split}
t^{p}_0 =  k^{-\frac{1}{2}}\lambda - \frac{7}{4}k^{-1} + \frac{\sqrt{k}}{\sqrt{k}+\beta\sqrt{vD}}k^{-1} + \ell + O\left(\lambda^{-1}\right) \quad \mbox{as} \quad \lambda \gg 1. 
\end{split}
\end{equation}
\end{theorem}
\begin{proof}
We first study the asymptotic behavior of the positive solution to $P(t)=0$ in \eqref{cubicpoly}.
\begin{proposition}\label{ellzero}
The positive solution $t_0>0$ to 
\begin{equation}\label{betazeroeq}
- k  - \frac{3}{2} t^{-1} + \lambda^2 t^{-2} - \frac{2\beta v D}{x_{c_3} + \beta v D t} = 0
\end{equation}
satisfies   
\begin{equation}\label{ellzeroasymp}
t_0 =  k^{-\frac{1}{2}}\lambda - \frac{7}{4}k^{-1} + \frac{\sqrt{k}}{\sqrt{k}+\beta\sqrt{vD}}k^{-1}  + O\left(\lambda^{-1}\right) \quad \mbox{as} \quad \lambda \gg 1. 
\end{equation}
\end{proposition}
\begin{proof}
Define $\tilde{t}>0$ as the positive solution to
$$
-k  - \frac{3}{2} t^{-1} + \lambda^2 t^{-2} =0.
$$
Then we obtain
\begin{equation}\label{betazeroasymp}
\begin{split}
\tilde{t} = \frac{1}{2k} \left( \left[ 4k\lambda^2 + \left(\frac{3}{2}\right)^2 \right]^\frac{1}{2} - \frac{3}{2}  \right) 
= k^{-\frac{1}{2}}\lambda - \frac{3}{4}k^{-1} + O\left( \lambda^{-1} \right) 
\end{split}
\end{equation}
as $\lambda \gg 1$.
Set $t_0 := \tilde{t} + \varepsilon$ with $\varepsilon = o(\lambda)$ as $\lambda \gg 1$.
By \eqref{betazeroeq}, we obtain
\begin{equation}\label{P(t)expan}
\begin{split}
P(t_0) &= - k  - \frac{3}{2} \left(\tilde{t} + \varepsilon\right)^{-1} + \lambda^2 \left(\tilde{t} + \varepsilon\right)^{-2} 
- 2\beta v D  \left( x_{c_3} + \beta v D [\tilde{t} + \varepsilon]\right)^{-1} \\
 &=
 -2\beta v D \left( x_{c_3} + \beta v D \tilde{t} \right)^{-1}  + \left[ -2 \lambda^2 \tilde{t}^{-3} + \frac{3}{2}\tilde{t}^{-2} + 2(\beta v D)^2  \left( x_{c_3} + \beta v D \tilde{t} \right)^{-2}\right]\varepsilon +  O\left( \lambda^{-2}\varepsilon^2 \right) \\
  &= 0.
\end{split}
\end{equation}
This together with  \eqref{defk} and \eqref{betazeroasymp} implies that
\begin{equation}
\begin{split}
\varepsilon &= \frac{2\beta v D}{ x_{c_3} + \beta v D \tilde{t}} \times
\frac{1}{-2 \lambda^2 \tilde{t}^{-3} + \frac{3}{2}\tilde{t}^{-2} + 2(\beta v D)^2  \left( x_{c_3} + \beta v D \tilde{t} \right)^{-2}} + \cdots \\
 &= -\frac{\beta v D}{ x_{c_3} + \beta v D \tilde{t}} \times \frac{\tilde{t}^3}{\lambda^2} +  O\left( \lambda^{-1} \right) \quad \mbox{as} \quad \lambda \gg 1.
\end{split}
\end{equation}
Note that $x_{c_3} \sim \sqrt{vD} \lambda$. 
By \eqref{betazeroasymp}, the dominant part of $\varepsilon$ is 
\begin{equation}
\begin{split}
-\frac{\beta v D}{ x_{c_3} + \beta v D \tilde{t}} \times \frac{\tilde{t}^3}{\lambda^2}  &= -\frac{\beta v D}{ \sqrt{vD} + \beta v D k^{-\frac{1}{2}}} k^{-\frac{3}{2}} +  O\left( \lambda^{-1} \right)\\
&= -\frac{\beta\sqrt{vD}}{\sqrt{k} + \beta\sqrt{vD}}k^{-1} 
+  O\left( \lambda^{-1} \right) \quad \mbox{as} \quad \lambda \gg 1.
\end{split}
\end{equation}
Then we obtain the asymptotic expansion 
\begin{equation}\label{peaktime formula_zerolifetime}
\begin{split}
t_0 &= \tilde{t} + \varepsilon 
=  k^{-\frac{1}{2}}\lambda - \frac{3}{4}k^{-1} - \frac{\beta\sqrt{vD}}{\sqrt{k} + \beta\sqrt{vD}}k^{-1} +  O\left( \lambda^{-1} \right) \\
&=  k^{-\frac{1}{2}}\lambda - \frac{7}{4}k^{-1} + \frac{\sqrt{k}}{\sqrt{k}+\beta\sqrt{vD}}k^{-1}  +  O\left( \lambda^{-1} \right) \quad \mbox{as} \quad \lambda \gg 1,
\end{split}
\end{equation}
and the proof of Proposition \ref{ellzero} is complete.
\end{proof}
Set $t^p_0 = t_0 + \varepsilon$  with $\varepsilon = o(\lambda)$ as $\lambda \gg 1$. We find the asymptotic behavior of $\varepsilon$ from the equation \eqref{ellsmalldef}
\begin{equation}\label{tripleP}
\begin{split}
P(t^p_0) - \ell \left( P^{'}(t^p_0) + P(t^p_0)^2  \right) = 0.
\end{split}
\end{equation}
In a similar way as in \eqref{P(t)expan}, we obtain
\begin{equation}\label{P(tp)expan1}
\begin{split}
P(t^p_0) &= - k  - \frac{3}{2} \left(t_0 + \varepsilon\right)^{-1} + \lambda^2 \left(t_0 + \varepsilon\right)^{-2} 
- 2\beta v D  \left( x_{c_3} + \beta v D [t_0 + \varepsilon]\right)^{-1} \\
 &=\left[ -2 \lambda^2 t_0^{-3} + \frac{3}{2}t_0^{-2} + 2(\beta v D)^2  \left( x_{c_3} + \beta v D t_0 \right)^{-2}\right]\varepsilon +  O\left( \lambda^{-2}\varepsilon^2 \right) \\
  &= -2 \lambda^2 t_0^{-3}\varepsilon + O\left( \lambda^{-2}\varepsilon \right) + O\left( \lambda^{-2}\varepsilon^2 \right)
\end{split}
\end{equation}
and
\begin{equation}\label{P(tp)expan2}
\begin{split}
P^{'}(t^p_0) &= \frac{3}{2} \left(t_0 + \varepsilon\right)^{-2} -2 \lambda^2 \left(t_0 + \varepsilon\right)^{-3} 
+ 2 (\beta v D)^2  \left( x_{c_3} + \beta v D [t_0 + \varepsilon]\right)^{-2} \\
 &=   -2 \lambda^2 t_0^{-3} + \frac{3}{2}t_0^{-2} + 2(\beta v D)^2  \left( x_{c_3} + \beta v D t_0 \right)^{-2} +O\left( \lambda^{-2}\varepsilon \right) \\
 &= -2 \lambda^2 t_0^{-3} +  O\left( \lambda^{-2} \right) + O\left( \lambda^{-2}\varepsilon \right)
\end{split}
\end{equation}
as $\lambda \gg 1$. Since 
$$
P(t^p_0)^2 =  O\left( \lambda^{-2}\varepsilon^2 \right) \quad \mbox{as} \quad \lambda \gg 1,
$$
we obtain from \eqref{tripleP} that
$$
2 \lambda^2 t_0^{-3}\varepsilon   -2 \ell \lambda^2 t_0^{-3} =  O\left( \lambda^{-2} \right) + O\left( \lambda^{-2}\varepsilon \right)+ O\left( \lambda^{-2}\varepsilon^2 \right).
$$
Then we have 
$$
\varepsilon = \ell + O\left( \lambda^{-1} \right)  \quad \mbox{as} \quad \lambda \gg 1.
$$
This together with \eqref{ellzeroasymp} implies
\begin{equation*}
\begin{split}
t^{p}_0 &= t_0 + \varepsilon \\
&=k^{-\frac{1}{2}}\lambda - \frac{7}{4}k^{-1} + \frac{\sqrt{k}}{\sqrt{k}+\beta\sqrt{vD}}k^{-1} + \ell + O\left(\lambda^{-1}\right) \quad \mbox{as} \quad \lambda \gg 1,
\end{split}
\end{equation*}
and the proof of Theorem \ref{thm_lzero} is complete.
\end{proof}
\begin{comment}
\begin{remark} In Theorem \ref{thm_lzero}, we can see the effect on $\beta>0$ in the boundary condition in \eqref{ue_sys} and  \eqref{um_sys}.
\begin{enumerate}
    \item When $\beta =0$, we obtain
    \begin{equation}\label{betazerocase}
        \begin{split}
        t^{p}_0 = k^{-\frac{1}{2}}\lambda - \frac{3}{4}k^{-1}
        +  O\left( \lambda^{-1} \right), \; \mbox{and it can be also derived from} \quad
        k t^2 + \frac{3}{2} t - \lambda^2 =0.
        \end{split}
\end{equation}
    \item When $\beta =\infty$, we obtain
    \begin{equation}\label{betainftycase}
        \begin{split}
        t^{p}_0 = k^{-\frac{1}{2}}\lambda - \frac{7}{4}k^{-1} +  O\left( \lambda^{-1} \right), \; \mbox{and it can be also derived from} \quad
        k t^2 + \frac{7}{2} t - \lambda^2 =0.
        \end{split}
\end{equation}
\end{enumerate}
\red{It is worth noting that case 2 never happens physically because $\beta=b/D$ and $b$ bounds to 1. If $D$ approaches 0, the model, coupled equations (1.1) and (1.2), does not work.}
\end{remark}
\end{comment}
%%%%%%%%%%%%%%%%%%%%%%%%%%%%%%%%%%%%%%%%%%%%%%%%%%%%%%%%%%%%%
\subsection{Case $\ell \gg 1$.}
Suppose $\ell \neq 0$. Since the solution $U_m$ to \eqref{um_sys} satisfies 
\begin{equation*}
\begin{split}
U_m(t) &= \int_0^t \ell^{-1}e^{-\frac{t-s}{\ell}} u_m (s)  \,{\rm d}s,  \\
\partial_t U_m(t)&= \ell^{-1} u_m(t) + \int_0^t 
(-\ell^{-2})e^{-\frac{t-s}{\ell}} u_m (s) \,{\rm d}s,
\end{split}
\end{equation*}
we find the critical point of $U_m$ by the positive solution to 
\begin{equation*}
\begin{split}
u_m(t) = \ell^{-1} \int_0^t 
e^{-\frac{t-s}{\ell}} u_m (s) \,{\rm d}s.
\end{split}
\end{equation*}
By replacing $u_m$ as $u^a_m$ in \eqref{defIt}, we obtain the peak time approximately by the positive solution to
\begin{equation}\label{t-deri_Um}
\begin{split}
 u^a_m (t) = \ell^{-1} \int_0^t  e^{-\frac{t-s}{\ell}} u^a_m (s) \,{\rm d}s.
\end{split}
\end{equation}
Write $u_m^a$ in \eqref{defIt} as
$$
 u^a_m(t) = e^{-(kt+\lambda^2 t^{-1})} f(t) \quad \mbox{with} \quad f(t) := C_0
 t^{-\frac{3}{2}}\left(\frac{x_{c_3}}{x_{c_3}+\beta v D t}\right)^2.
$$
In the next Lemma \ref{integralest}, one can see the asymptotic behavior of the time integration of $u_m^a$, which is determined by the value at $t = \lambda k^{-1/2}$ (see \cite[Theorem 2.3 ]{Chen2025}).
\begin{lemma}\label{integralest}
Assume $t>\lambda k^{-\frac{1}{2}}$.
Then $u^a_m$ of \eqref{defIt} satisfies
\begin{equation}
\label{integralbehavior}
\int_0^t  e^{-(ks+\lambda^2 s^{-1})} f(s) \; {\rm d}s = k^{-\frac{3}{4}} \left( \pi \lambda \right)^\frac{1}{2} 
 e^{-2 \lambda \sqrt{k}} f( \lambda k^{-\frac{1}{2}})  + O\left( \lambda^{-\frac{3}{2}} e^{-2 \lambda \sqrt{k} }  f( \lambda k^{-\frac{1}{2}}) \right) 
\end{equation}
as $\lambda \gg 1$, where $k>0$ and $\lambda$ are as in \eqref{defk}.
\end{lemma}
By replacing $k$ as $(k-\ell)^{-1/2}$ in Lemma  \ref{integralest},
for 
\begin{equation}
\label{ellcondition1}
t>\lambda (k-\ell^{-1})^{-\frac{1}{2}} \quad \mbox{and} \quad \ell > k^{-1},
\end{equation}
we obtain 
\begin{equation*}
\begin{split}
\int_0^t  e^{-(t-s)\ell^{-1}} u^a_m (s) \,{\rm d}s &= 
 e^{-t \ell^{-1}} \int_0^t e^{s\ell^{-1}} u^a_m (s) \,{\rm d}s \\
&=  e^{-t \ell^{-1}} \int_0^t  e^{-\left([k-\ell^{-1}]s + \lambda^2 s^{-1} \right)} f(s) \,{\rm d}s \\
&\sim e^{-t \ell^{-1}} (k-\ell^{-1})^{-\frac{3}{4}} \left( \pi \lambda \right)^\frac{1}{2} 
 e^{-2 \lambda \sqrt{k-\ell^{-1}}} f( \lambda (k-\ell^{-1})^{-\frac{1}{2}}) 
\end{split}
\end{equation*}
This together with \eqref{t-deri_Um} implies
$$
u_m^a(t) = e^{-(kt+\lambda^2 t^{-1})} f(t) \sim \ell^{-1} e^{-t \ell^{-1}}
(k-\ell^{-1})^{-\frac{3}{4}} \left( \pi \lambda \right)^\frac{1}{2} 
 e^{-2 \lambda \sqrt{k-\ell^{-1}}} f( \lambda (k-\ell^{-1})^{-\frac{1}{2}}). 
$$
Then we obtain the equation for the peak time approximately
\begin{equation}\label{NonlinearScheme}
\begin{split}
 e^{-\frac{(\sqrt{k-\ell^{-1}}t-\lambda)^2}{t}} = \pi^{\frac{1}{2}} \ell^{-1} \lambda^{-1} t^{\frac{3}{2}} \left(\frac{x_{c_3}+\beta vDt }{x_{c_3}+\beta vD \lambda (k-\ell^{-1})^{-\frac{1}{2}}}\right)^2.
\end{split}
\end{equation}
If the fluorescence lifetime $\ell>0$ satisfies
\begin{equation}\label{ellcondition2}
\begin{split}
\ell>\pi^{\frac{1}{2}}(k-\ell^{-1})^{-\frac{3}{4}}\lambda^{\frac{1}{2}},
\end{split}
\end{equation}
then the unique existence of the positive solution to \eqref{NonlinearScheme} is guaranteed.
\begin{definition}\label{def_linfty}
Assume \eqref{ellcondition1} and \eqref{ellcondition2}. We define approximate peak time $t^{p}_\infty$ by the positive solution to \eqref{NonlinearScheme} in the case of a large fluorescence lifetime.
\end{definition}
In the next Theorem \ref{thm_linfty}, we derive the asymptotic expansion of the approximate peak time $t^p_\infty$ when $\lambda \gg 1$.
\begin{theorem}\label{thm_linfty}
The approximate peak time for a large fluorescence lifetime satisfies
\begin{equation}\label{asympofpeak2}
\begin{split}
t^{p}_\infty = (k-\ell^{-1})^{-\frac{1}{2}} \lambda + (k-\ell^{-1})^{-\frac{3}{4}} \alpha_\lambda \lambda^\frac{1}{2} 
+ O\left( \lambda^{-\frac{1}{2}} \right) \quad \mbox{as} \quad \lambda \gg 1, 
\end{split}
\end{equation}
where $\alpha_\lambda>0$ is
\begin{equation}\label{defalpha}
\begin{split}
\alpha_\lambda := \left(-\log\left[ (\pi^{\frac{1}{2}} \ell^{-1} (k-\ell^{-1})^{-\frac{3}{4}}) \lambda^\frac{1}{2} \right] \right)^\frac{1}{2} < \infty.
\end{split}
\end{equation}
\end{theorem}
\begin{proof}
Set
\begin{equation}\label{peaktimeapp}
t = (k-\ell^{-1})^{-\frac{1}{2}} \lambda + \varepsilon, \quad \mbox{where} \quad \varepsilon = o (\lambda) \quad \mbox{as} \quad \lambda \gg 1.
\end{equation}
Since
$$
\left(\frac{x_{c_3}+\beta vDt }{x_{c_3}+\beta vD \lambda (k-\ell^{-1})^{-\frac{1}{2}}}\right)^2 = 1 + O\left( \lambda^{-1} \varepsilon \right),
$$
we obtain from \eqref{NonlinearScheme} that
\begin{equation*}
\begin{split}
e^{-\frac{(k-\ell^{-1})\varepsilon^2}{(k-\ell^{-1})^{-1/2} \lambda + \varepsilon}} &=
\pi^{\frac{1}{2}} \ell^{-1} \lambda^{-1} t^{\frac{3}{2}} \left[ 1+
O\left( \lambda^{-1} \varepsilon \right) \right] \\
 &=\pi^{\frac{1}{2}} \ell^{-1} \lambda^{-1}   \left( [(k-\ell^{-1})^{-\frac{1}{2}} \lambda ]^{\frac{3}{2}} + O\left( \lambda^\frac{1}{2} \right) \right)  \left( 1+
O\left( \lambda^{-1} \varepsilon \right) \right) \\
&= \alpha \lambda^\frac{1}{2} \left[ 1 + O\left( \lambda^{-1} \right) \right] \quad \mbox{with} \quad 
\alpha:= \pi^{\frac{1}{2}} \ell^{-1} (k-\ell^{-1})^{-\frac{3}{4}}
\end{split}
\end{equation*}
as $\lambda \gg 1$.
By
\begin{equation*}
\begin{split}
-\frac{(k-\ell^{-1})\varepsilon^2}{(k-\ell^{-1})^{-1/2} \lambda + \varepsilon} = -(k-\ell^{-1})^\frac{3}{2} \varepsilon^2 \left[ \lambda^{-1} + O\left( \lambda^{-2}\varepsilon \right) \right] \quad \mbox{as} \quad \lambda \gg 1,
\end{split}
\end{equation*}
and
\begin{equation*}
\begin{split}
\log\left( \alpha \lambda^\frac{1}{2} \left[ 1 + O\left( \lambda^{-1} \right) \right]\right) = \log\left(\alpha \lambda^\frac{1}{2} \right) +  O\left( \lambda^{-1} \right)
\quad \mbox{as} \quad \lambda \gg 1,
\end{split}
\end{equation*}
we obtain 
\begin{equation}
\varepsilon^2 = (k-\ell^{-1})^{-\frac{3}{2}} \left( -\log\left[ \alpha \lambda^\frac{1}{2} \right] \right) \lambda + O\left( 1 \right) \quad \mbox{as} \quad \lambda \gg 1.
\end{equation}
Since $0<\alpha \lambda^{1/2} <1$, we have 
$$
\varepsilon =  (k-\ell^{-1})^{-\frac{3}{4}} \left( -\log\left[ \alpha \lambda^\frac{1}{2} \right] \right)^\frac{1}{2} \lambda^\frac{1}{2} + O\left( \lambda^{-\frac{1}{2}} \right) \quad \mbox{as} \quad \lambda \gg 1.
$$
Then we have the asymptotic expansion of the approximate peak time
\begin{equation}\label{eq_peaktime}
\begin{split}
t &= (k-\ell^{-1})^{-\frac{1}{2}} \lambda + \varepsilon \\
&=(k-\ell^{-1})^{-\frac{1}{2}} \lambda + (k-\ell^{-1})^{-\frac{3}{4}} \left( -\log\left[ (\pi^{\frac{1}{2}} \ell^{-1} (k-\ell^{-1})^{-\frac{3}{4}}) \lambda^\frac{1}{2} \right] \right)^\frac{1}{2} \lambda^\frac{1}{2} 
\end{split}
\end{equation}
as $\lambda \gg 1$, and the proof the Theorem \ref{thm_linfty} is complete.
\end{proof}

\section{Numerical verification}\label{sec_numeric}
In this section,  we numerically verify the accuracy of the asymptotic of the approximate peak time by comparing to the peak time for the cases $\ell \ll 1$ and $\ell \gg 1$. If no otherwise specified, we set  S-D pair $\{x_{d},\,x_{s}\}=\{(14,\,10,\,0),\,(6,\,10,\,0)\}$, the target location $x_{c}=(10,\,10,\,x_{c_3})$ with $x_{c_3}=20\,{\rm mm}$ and 
\begin{equation}\label{phys_para}
v=0.219\, {\rm mm/ps},\quad D=1/3\,{\rm mm},\quad \mu_a=0.1\,{\rm mm^{-1}},\quad \beta=0.5493\, {\rm mm^{-1}},
\end{equation}
which are typical values in biological tissues. 
\begin{comment}
In Figure \ref{fig_tps_diffpara},
we show the numerical results for fixed S-D pair $\{x_{d},\,x_{s}\}=\{(14,\,10,\,0),\,(6,\,10,\,0)\}$, the projected location of the target $x_{c}=(10,\,10,\,x_{c_3})$ and changed $x_{c_3},\,D,\,\mu_a,\,\ell$.
\end{comment}
By \eqref{asympofpeak1}, define $t_{peak}^s$ as the asymptotic profile of the approximate peak time for small fluorescence lifetime 
\begin{equation}\label{asymptotic1}
\begin{split}
t_{peak}^s :=  k^{-\frac{1}{2}}\lambda - \frac{7}{4}k^{-1} + \frac{\sqrt{k}}{\sqrt{k}+\beta\sqrt{vD}}k^{-1} + \ell,
\end{split}
\end{equation}
where $k>0$ and $\lambda>0$ are is as in \eqref{defk}.
In Figure \ref{fig_tpszero_diffpara}, we can see the behavior of peak time $t_{peak}$ and $t_{peak}^s$ with their relative errors depending on $\ell>0$, $\mu_a>0$ and $x_{c_3}>0$.

\begin{figure}[htp]
\centering
\begin{tabular}{lll}
(a) & (b) & (c) \\
\includegraphics[width=0.33\textwidth]{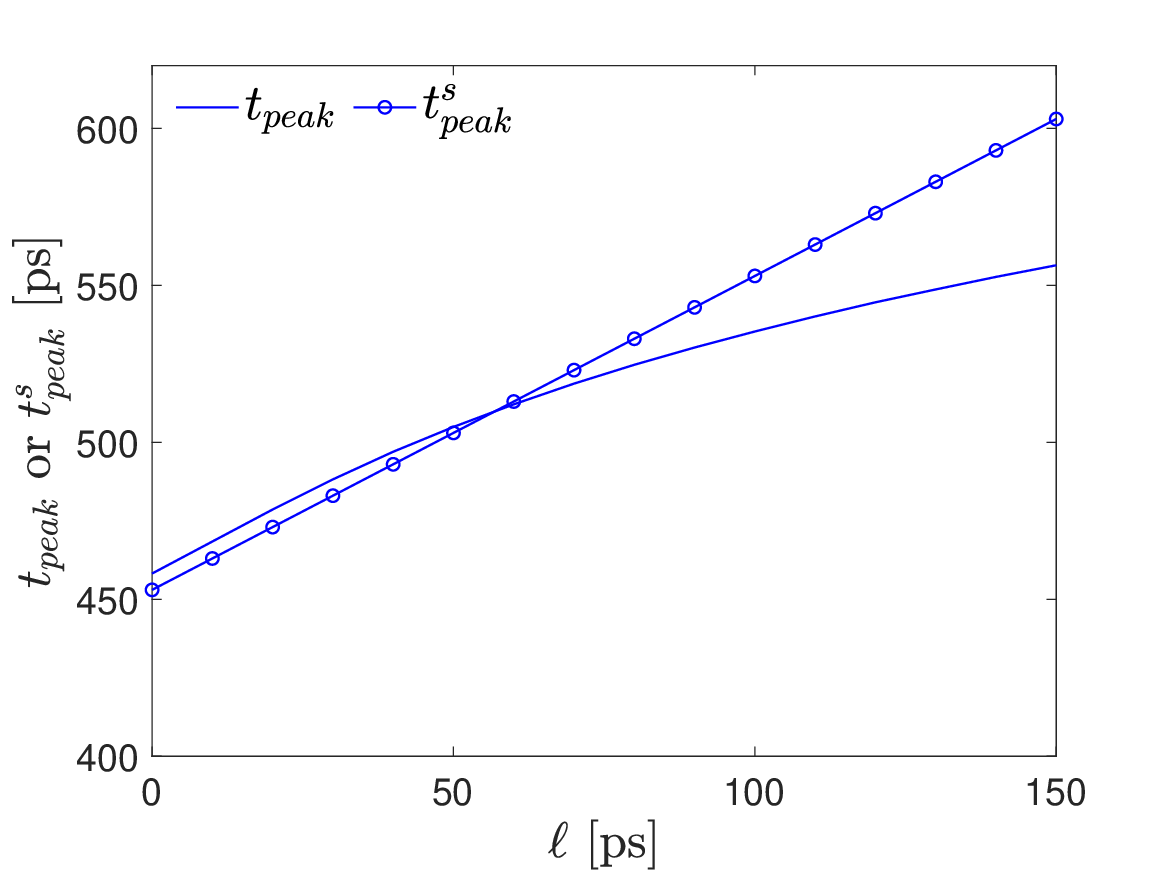}
& \includegraphics[width=0.33\textwidth]{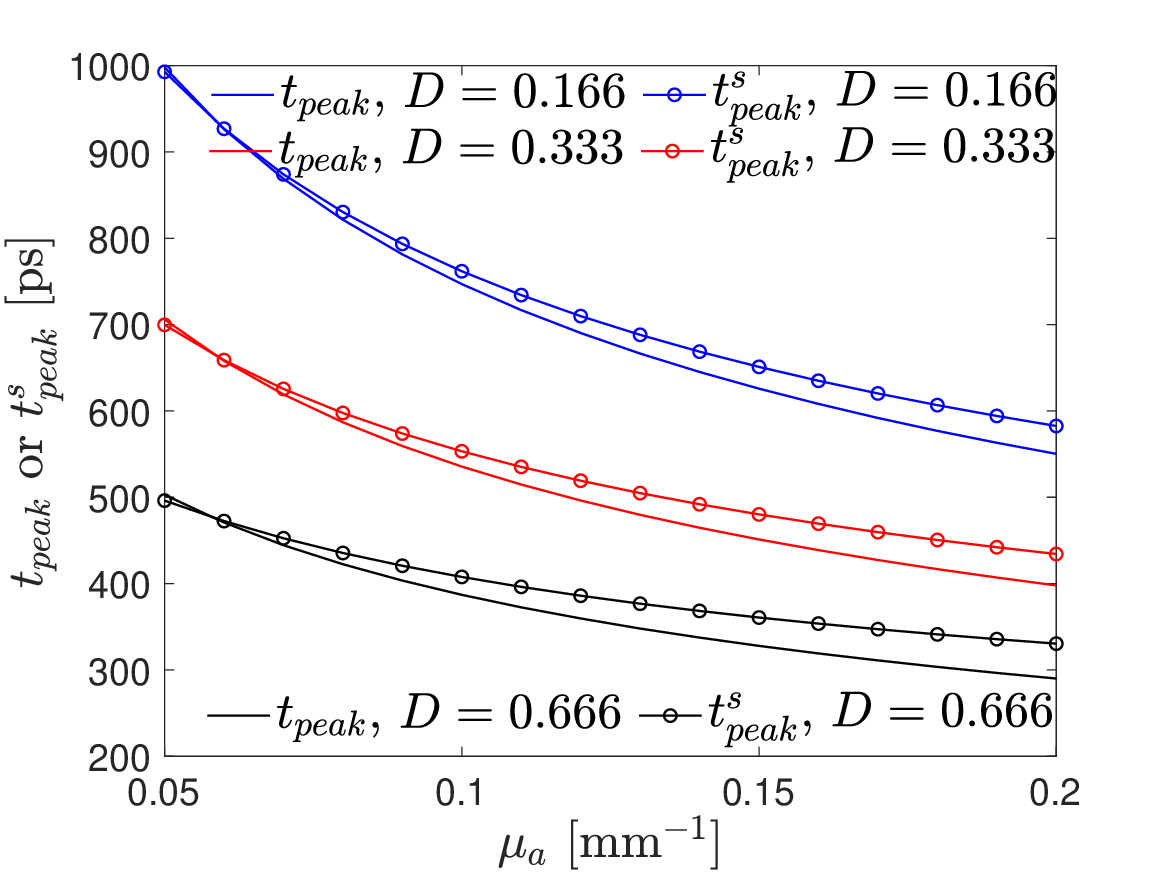}
& \includegraphics[width=0.33\textwidth]{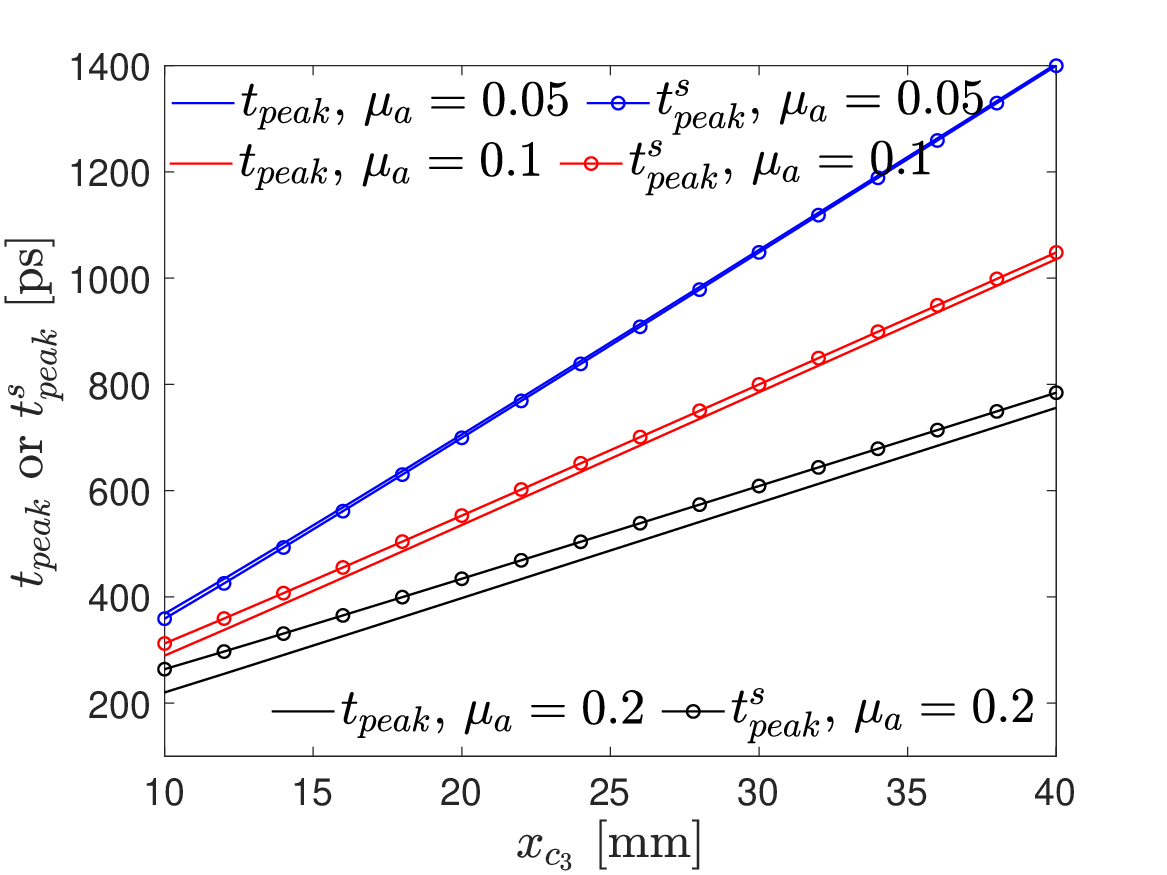}
\end{tabular}
\caption{Peak time $t_{peak}$ and the asymptotic $t_{peak}^s$. For (b) and (c), we set $\ell = 100 \; {\rm ps}$.}
\label{fig_tpszero_diffpara}
\end{figure}

In Figure \ref{fig_tps_diffpara}, we can see the behavior of peak time $t_{peak}$ and $t_{peak}^{l_0}$ which is the positive solution to
\begin{equation*}
\begin{split}
 e^{-\frac{(\sqrt{k-\ell^{-1}}t-\lambda)^2}{t}} = \pi^{\frac{1}{2}} \ell^{-1} \lambda^{-1} t^{\frac{3}{2}} \left(\frac{x_{c_3}+\beta vDt }{x_{c_3}+\beta vD \lambda (k-\ell^{-1})^{-\frac{1}{2}}}\right)^2
\end{split}
\end{equation*}
depending on $l>0$, $\mu_a>0$ and $x_{c_3}>0$. By \eqref{asympofpeak2}, define $t_{peak}^l$ as the asymptotic profile of the approximate peak time for large fluorescence lifetime
\begin{equation}\label{asymptotic2}
\begin{split}
t_{peak}^l = (k-\ell^{-1})^{-\frac{1}{2}} \lambda + (k-\ell^{-1})^{-\frac{3}{4}} \alpha_\lambda \lambda^\frac{1}{2},
\end{split}
\end{equation}
where $\alpha_\lambda>0$ is as in \eqref{defalpha}.
In Figure \ref{fig_tps_linear_diffpara}, we can see the behavior of peak time $t_{peak}$ and $t_{peak}^{l}$ depending on $l>0$, $\mu_a>0$ and $x_{c_3}>0$.

\begin{figure}[htp]
\centering
\begin{tabular}{lll}
(a) & (b) & (c) \\
\includegraphics[width=0.33\textwidth]{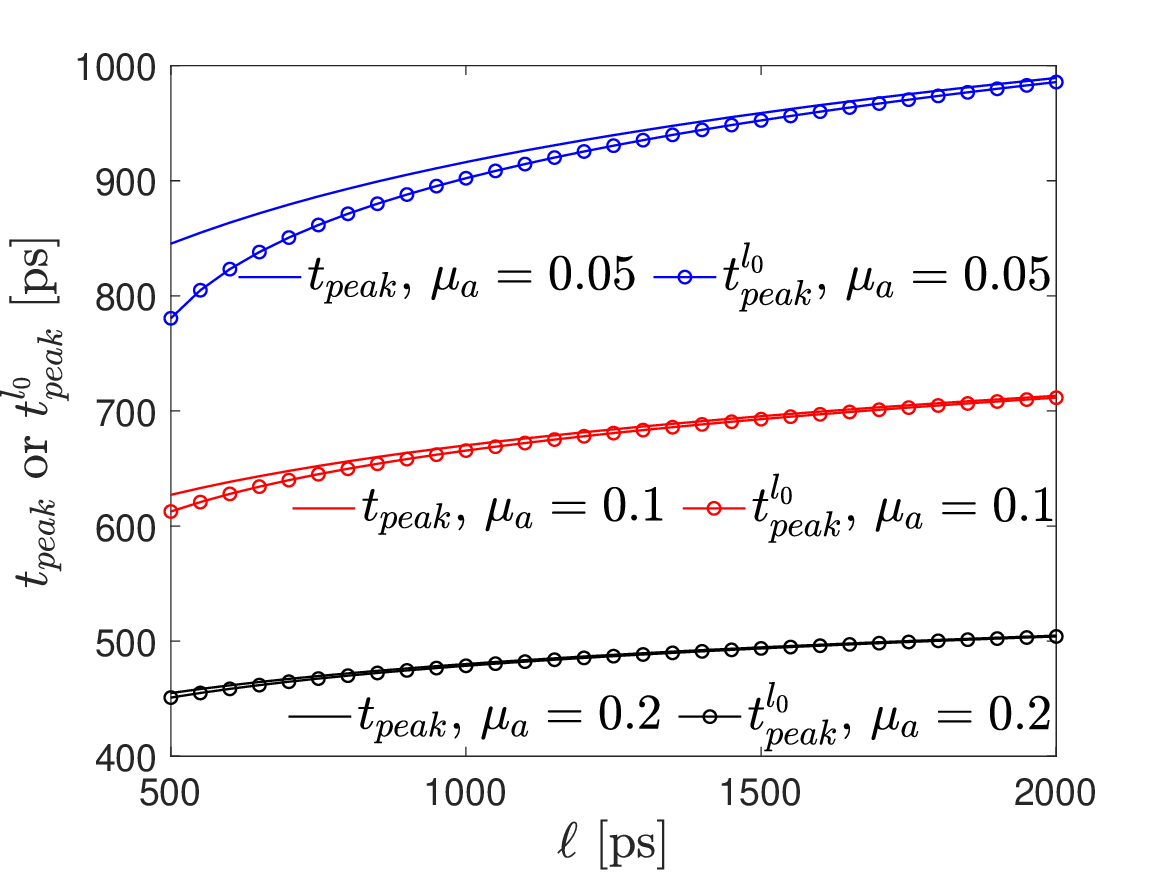}
& \includegraphics[width=0.33\textwidth]{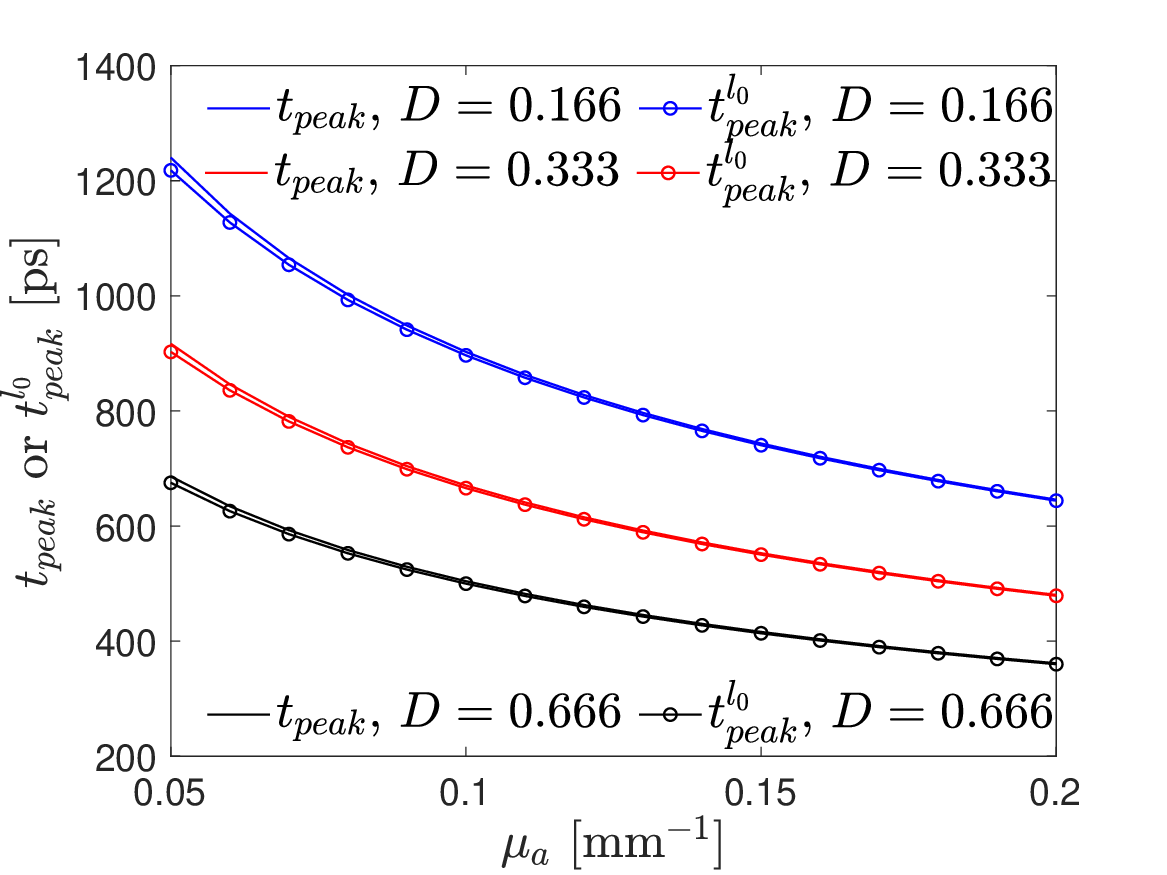}
& \includegraphics[width=0.33\textwidth]{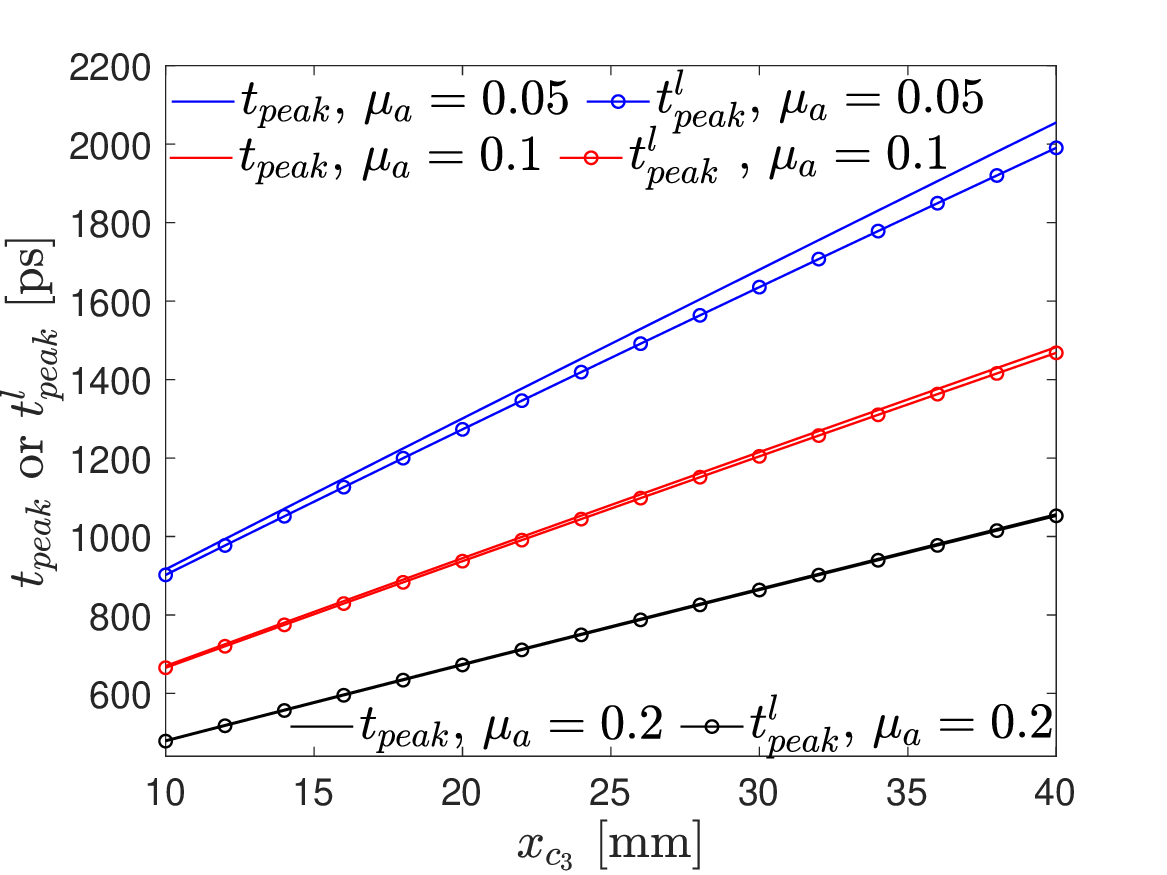}
\end{tabular}
\caption{Peak time $t_{peak}$ and the approximate peak time $t_{peak}^{l_0}$. For (b) and (c), we set $\ell = 1000 \; {\rm ps}$.}
\label{fig_tps_diffpara}
\end{figure}

\begin{figure}[htp]
\centering
\begin{tabular}{lll}
(a) & (b) & (c) \\
\includegraphics[width=0.33\textwidth]{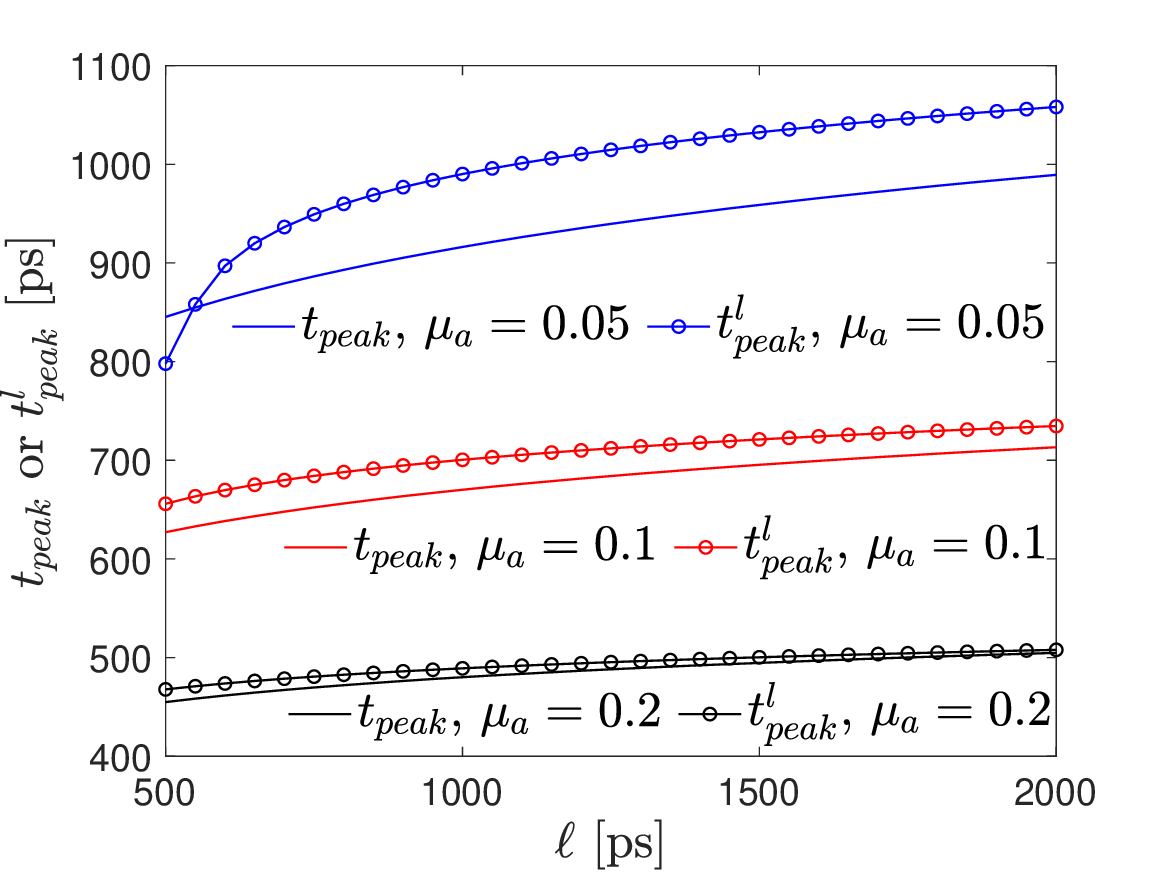}
& \includegraphics[width=0.33\textwidth]{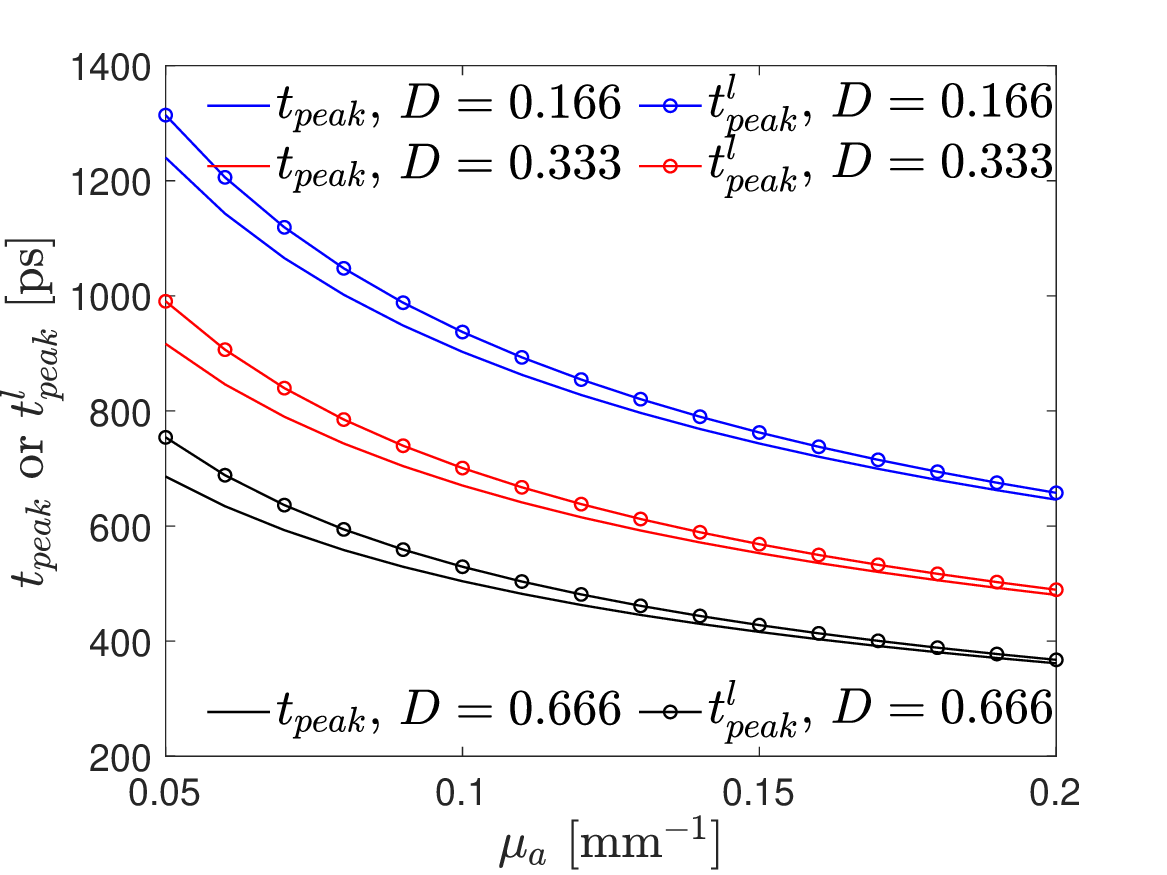}
& \includegraphics[width=0.33\textwidth]{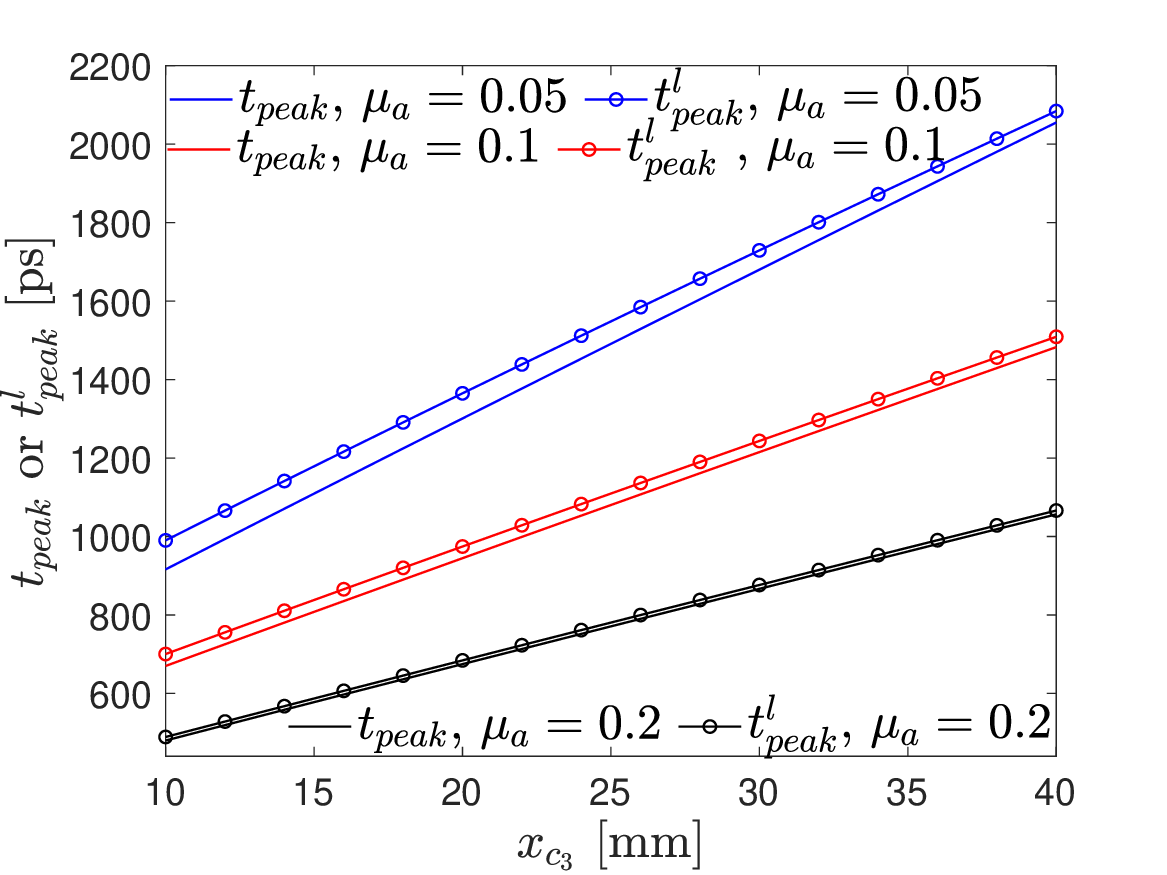}
\end{tabular}
\caption{Peak time $t_{peak}$ and the approximate peak time $t_{peak}^{l}$. For (b) and (c), we set $\ell = 1000 \; {\rm ps}$.}
\label{fig_tps_linear_diffpara} 
\end{figure}

\begin{comment}
In addition, we numerically verify the accuracy of an approximate peak time defined by 
\eqref{sim_peak_formula}. See Figure \ref{fig_tps_diffpara1}.

\begin{figure}[htp]
\centering
\begin{tabular}{lll}
(a) & (b) & (c) \\
\includegraphics[width=0.33\textwidth]{figures1/tps1_lt_mua_s.eps}
& \includegraphics[width=0.33\textwidth]{figures1/tps1_mua_D_s.eps}
& \includegraphics[width=0.33\textwidth]{figures1/tps1_xc3_mua_s.eps}
\end{tabular}
\caption{Peak time and approximate peak time for different physical parameters}
\label{fig_tps_diffpara1}
\end{figure}

Based on \eqref{eq_peaktime}, we numerically verify the accuracy of two approximate peak times
$$t_{peak}^{a1}=  k^{-\frac{1}{2}} \lambda + k^{-\frac{3}{4}}\lambda^\frac{1}{2},\quad
t_{peak}^{a2}=k^{-\frac{1}{2}} \lambda + k^{-\frac{3}{4}}\lambda^\frac{1}{2} - \textcolor{red}{\frac{1}{2}
 \pi^{\frac{1}{2}}k^{-\frac{3}{2}}
 \ell^{-1} \lambda}.$$
The results are shown in Figure \ref{fig_tps_diffpara2}.

 \begin{figure}[htp]
\centering
\begin{tabular}{lll}
(a) & (b) & (c) \\
\includegraphics[width=0.33\textwidth]{figures1/tps12_lt_mua.eps}
& \includegraphics[width=0.33\textwidth]{figures1/tps12_mua_D.eps}
& \includegraphics[width=0.33\textwidth]{figures1/tps12_xc3_mua.eps}
\end{tabular}
\caption{Peak time and approximate peak time for different physical parameters}
\label{fig_tps_diffpara2}
\end{figure}
\end{comment}

\section{Inversion scheme}\label{sec_inversion}
In this section, we propose a simple inversion scheme by using three S-D pairs. In each S-D pair $(x_s, x_d)$, we measure the peak time, and as a consequence, we determine the distance parameter $\lambda>0$ in \eqref{defk}, which characterizes the radius of sphere at the center $(x_s+x_d)/2$, where the target $x_c$ lies on the sphere. Our scheme aims to determine three unknowns in the point target by three radii from three different S-D pairs.

Based on the asymptotic expansions for the peak time in Theorems \ref{thm_lzero} and \ref{thm_linfty}, we derive the asymptotic expansions for the distance parameter $\lambda>0$ by inverting the relations between the peak time $t>0$ and $\lambda>0$. 
When $\ell \ll 1$, by \eqref{asympofpeak1}, one can easily obtain
\begin{equation}\label{lambda_expan1}
\lambda = k^\frac{1}{2}t + \frac{7}{4}k^{-\frac{1}{2}} - 
\frac{\sqrt{k}}{\sqrt{k} + \beta \sqrt{vD}}k^{-\frac{1}{2}}
-\ell k^\frac{1}{2} + O\left( t^{-1} \right) 
\end{equation}
as $t \gg 1$. When $\ell \gg 1$, by \eqref{asympofpeak2}, 
\begin{equation*}
\begin{split}
\lambda &= (k-\ell^{-1})^\frac{1}{2} t -  (k-\ell^{-1})^{-\frac{1}{4}} \alpha_\lambda \lambda^\frac{1}{2} + O\left( t^{-\frac{1}{2}} \right) \\
&= (k-\ell^{-1})^\frac{1}{2} t -  (k-\ell^{-1})^{-\frac{1}{4}}\alpha_\lambda\left[ (k-\ell^{-1})^\frac{1}{2} t -  (k-\ell^{-1})^{-\frac{1}{4}} \alpha_\lambda \lambda^\frac{1}{2} + O\left( t^{-\frac{1}{2}} \right) \right]^\frac{1}{2}  +O\left( t^{-\frac{1}{2}} \right) \\
&=(k-\ell^{-1})^\frac{1}{2} t -  (k-\ell^{-1})^{-\frac{1}{4}}\alpha_\lambda\left[ (k-\ell^{-1})^\frac{1}{4} t^\frac{1}{2} - \frac{\alpha_\lambda}{2}(k-\ell)^{-\frac{1}{2}} t^{-\frac{1}{2}} \lambda^\frac{1}{2} + O\left( t^{-\frac{1}{2}} \right)
\right] + O\left( t^{-\frac{1}{2}} \right) \\
&=(k-\ell^{-1})^\frac{1}{2}t - \alpha_\lambda t^\frac{1}{2} + \frac{\alpha_\lambda^2}{2} (k-\ell^{-1})^{-\frac{1}{2}}+ O\left( t^{-\frac{1}{2}} \right)
\end{split}
\end{equation*}
as $t \gg 1$. Since
\begin{equation*}
\begin{split}
\alpha_\lambda &= \left(-\log\left[ (\pi^{\frac{1}{2}} \ell^{-1} (k-\ell^{-1})^{-\frac{3}{4}}) \lambda^\frac{1}{2} \right] \right)^\frac{1}{2} =\left(-\log\left[ (\pi^{\frac{1}{2}} \ell^{-1} (k-\ell^{-1})^{-\frac{1}{2}}) (t^\frac{1}{2} + O(1)) \right] \right)^\frac{1}{2}  \\
&= \tilde \alpha_t +  O\left( t^{-\frac{1}{2}} \right)\quad \mbox{with} \quad \tilde \alpha_t:=
 \left(-\log\left[ (\pi^{\frac{1}{2}} \ell^{-1} (k-\ell^{-1})^{-\frac{1}{2}}) t^\frac{1}{2} \right] \right)^\frac{1}{2} < \infty
\end{split}
\end{equation*}
as $t \gg 1$, we obtain
\begin{equation}\label{lambda_expan2}
\begin{split}
\lambda = (k-\ell^{-1})^\frac{1}{2}t - \tilde \alpha_t t^\frac{1}{2} + \frac{\tilde\alpha_t^2}{2} (k-\ell^{-1})^{-\frac{1}{2}} + O\left( t^{-\frac{1}{2}} \right)
\end{split}
\end{equation}
as $t \gg 1$.
By \eqref{lambda_expan1} and \eqref{lambda_expan2}, we calculate
the distance $\lambda>0$ from the peak time $t>0$
\begin{equation}\label{lambdaexpansion}
\lambda(t) = \left\{
\begin{array}{ll}
k^\frac{1}{2}t + \frac{7}{4}k^{-\frac{1}{2}} - 
\frac{\sqrt{k}}{\sqrt{k} + \beta \sqrt{vD}}k^{-\frac{1}{2}}
-\ell k^\frac{1}{2} &  \mbox{if} \quad \ell \ll 1, \vspace{5pt} \\ 
(k-\ell^{-1})^\frac{1}{2}t - \tilde \alpha_t t^\frac{1}{2} + \frac{\tilde \alpha_t^2}{2} (k-\ell^{-1})^{-\frac{1}{2}}& 
\mbox{if} \quad \ell > \pi^{\frac{1}{2}} (k-\ell^{-1})^{-\frac{1}{2}} t^\frac{1}{2},
\end{array}
\right.
\end{equation}
where
\begin{equation}\label{defalphatilde}
\tilde \alpha_t:=
 \left(-\log\left[ \pi^{\frac{1}{2}} \ell^{-1} (k-\ell^{-1})^{-\frac{1}{2}} t^\frac{1}{2} \right] \right)^\frac{1}{2}.
\end{equation}
By \eqref{defk}, we can see that the target $x_c$ is on the sphere 
\begin{equation}
\begin{split}
\left|x_c - \frac{x_d + x_s}{2}\right|^2 
= r^2 \quad
\mbox{with} \quad 
r := \sqrt{vD\lambda^2
- \frac{|x_d - x_s|^2}{4}}. 
\end{split}
\end{equation}
Given S-D pair $(x_s, x_d)$ , set the radius $r>0$ from the distance $\lambda$
\begin{equation}\label{def_radius}
\begin{split}
r(\lambda) = \sqrt{vD\lambda^2 - \frac{|x_d - x_s|^2}{4}}. 
\end{split}
\end{equation}
Set the initial S-D pair $(x_s, x_d)$ and obtain the radius $r>0$ from the equations \eqref{lambdaexpansion} and \eqref{def_radius}. As the second S-D pair, we move the initial S-D pair on the disk of radius $r>0$ with the center $(x_s+x_d)/2$ in the direction of $\theta_1$ and obtain the radius $r_1>0$. As the last S-D pair, we move the initial S-D pair on the same disk as before in the direction of $\theta_2$, where $\theta_2 \bot \theta_1$, and obtain the radius $r_2>0$. Then we obtain one tetrahedral with all known edges. 
By translating $(x_s+x_d)/2$ to the origin, we obtain the tetrahedral, which has the four vertices
\begin{equation}
\begin{split}
&O(0, 0, 0), \quad A(r \cos\theta_1,r \sin\theta_1,0 ), \quad
B(r \cos\theta_2,r \sin\theta_2,0 ), \quad
C(c_1,c_2,c_3) : \mbox{target location} \\
&\mbox{with} \quad \overline{OC}=r>0, \quad \overline{AC}=r_1>0, 
\quad \overline{BC}=r_2>0.
\end{split}
\end{equation}
See Figure \ref{fig_tetrahedron} for the tetrahedron.
Say $H_1$ and $H_2$ projection points on the lines 
$\overline{OA}$ and $\overline{OB}$ from the target vertex $C$, respectively. 

\begin{figure}[ht]
\centering
\includegraphics[width=.4\textwidth,height=0.25\textheight]{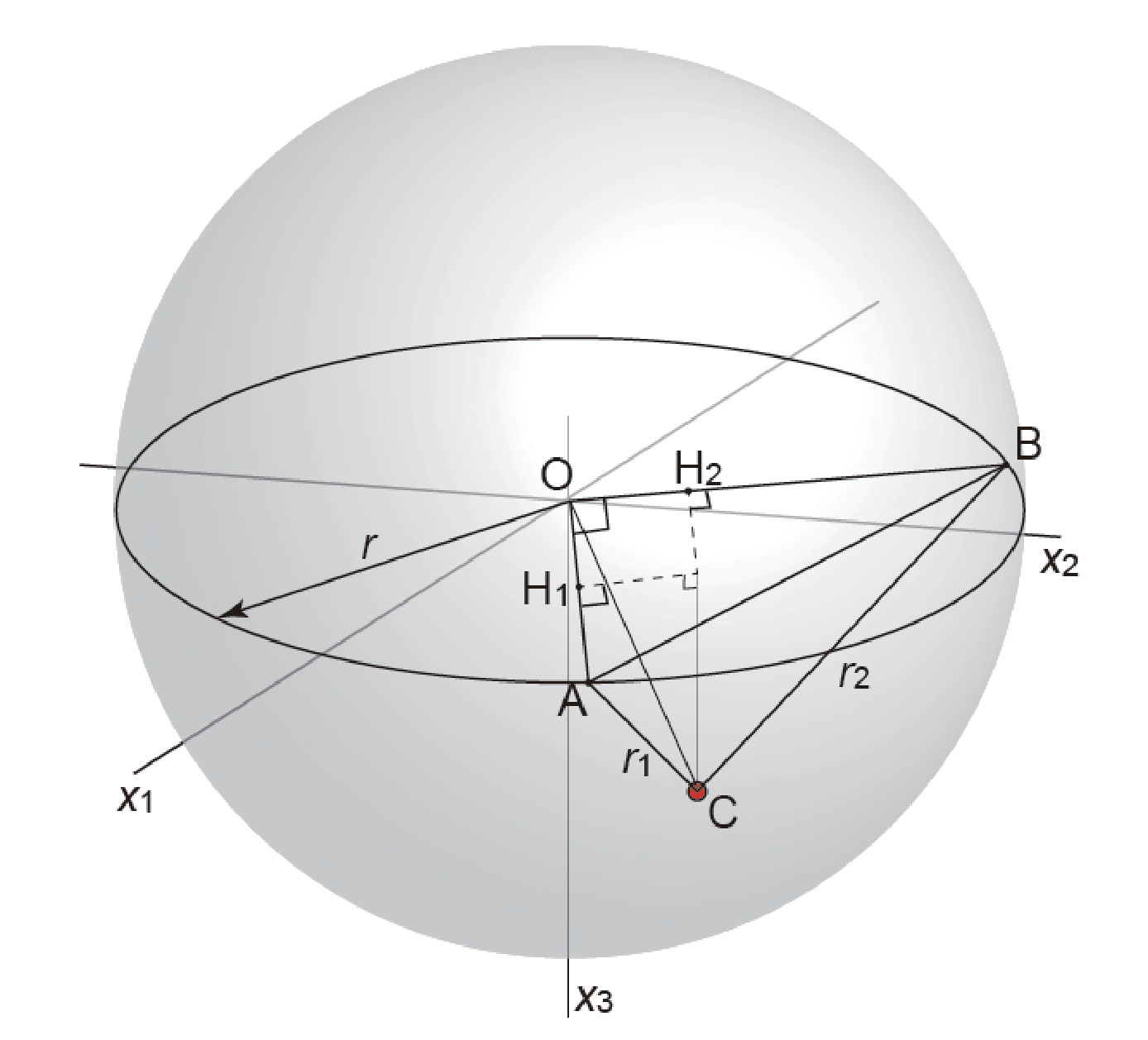}
\caption{Tetrahedron with four vertices $O,\;A,\;B$ and $C$. Point $A$ is the center point of the initial S-D pair. Points $B$ and $C$ are the center points of the second and third S-D pairs, respectively. Point $C$ is the target location.} 
\label{fig_tetrahedron}
\end{figure}

Write the $x,y$-coordinates of the vertex $C$ as
$$
c_1 = x \cos\theta_1 +  y \cos\theta_2, \;
c_2 = x \sin\theta_1 +  y \sin\theta_2 \quad \mbox{with}
\quad x = \overline{OH}_1, \; y = \overline{OH}_2.
$$
By the three-perpendiculars theorem, we obtain the relations
\begin{equation}
\begin{split}
&x^2 + \overline{CH}_1^2 = r^2,  \quad \overline{CH}_1 = \frac{r_1}{r} \sqrt{r^2 - \frac{1}{4}r_1^2}, \\
&y^2 + \overline{CH}_2^2 = r^2,  \quad \overline{CH}_2 = \frac{r_2}{r} \sqrt{r^2 - \frac{1}{4}r_2^2}, \\
&x^2 + c_3^2 = \overline{CH}_2^2 \quad \mbox{and} \quad y^2 + c_3^2 = \overline{CH}_1^2. 
\end{split}
\end{equation}
Then we have 
\begin{equation}
\begin{split}
&x = \frac{1}{2r} (2r^2-r_1^2), \quad y = \frac{1}{2r} (2r^2-r_2^2), \quad
c_3 = \sqrt{r_1^2+r_2^2-\frac{r_1^4+r_2^4}{4r^2}-r^2},
\end{split}
\end{equation}
and the target $C$ is determined by 
\begin{equation}\label{ex_form_recover}
\begin{split}
 c_1 &= \left( \frac{x_s+x_d}{2} \right)_x
+ \frac{1}{2r} (2r^2-r_1^2)\cos\theta_1 +
 \frac{1}{2r} (2r^2-r_2^2)\cos\theta_2 , \\
c_2 &= \left( \frac{x_s+x_d}{2} \right)_y
+  \frac{1}{2r} (2r^2-r_1^2)\sin\theta_1 +
 \frac{1}{2r} (2r^2-r_2^2)\sin\theta_2 , \\
c_3 &= \sqrt{r_1^2+r_2^2-\frac{r_1^4+r_2^4}{4r^2}-r^2}.
\end{split}
\end{equation}
\begin{comment}
\begin{remark} On the inversion scheme, one can select a quadrant, where the target lies in the quadrant by comparing the value of the radius $r_1$ to the initial radius $r$. If $r_1 \ge \sqrt{2} r$, then the increment $x$ in the direction $\theta_1$ is negative which implies there is the target in the opposite direction $-\theta_1$. 
\end{remark}
\end{comment}
%%%%%%%%%%%%%%%%%%%%%%%%%%%%%%%%%%%%%%%%%%%%%%%%%%%%%%%%%%%%%%%%%

\section{Numerical experiments}\label{sec_numericexp}

In this section, we apply the proposed reconstruction algorithm \eqref{ex_form_recover} to several numerical experiments. Since the accuracy of the peak time equations \eqref{asympofpeak1} and \eqref{asympofpeak2} has been numerically verified for different physical parameters in Section \ref{sec_numeric}, we will focus on showing the effectiveness of the reconstruction algorithm from the perspective of noisy measurements and different initial S-D pair.
In all experiments, we set the physical parameters as \eqref{phys_para}.  

Considering noise, specifically time jitters, contained in the measurements, we perturb $t_{peak}$ by
\begin{equation}\label{noisydata_add}
t_{peak}^\delta:=\left(1+\hat\delta\times(2\times {\rm rand(1)} -1)\right)\times t_{peak},
\end{equation}
where $\hat\delta$ is the relative noise level, and $\rm rand(1)$ generates a uniformly distributed random number on the interval $(0,\,1)$.
We compute the relative error of the reconstruction by the following formula:
\begin{equation}\label{equ_relerr_xc}
RelErr=\frac{|x_c-x_c^{inv}|}{|x_c|},
\end{equation}
where $x_c$ and $x_c^{inv}$ are the actual and reconstructed location of the target, respectively.

In the following, we consider the inverse problem for the cases $\ell\ll 1$ and $\ell\gg 1$, respectively. Without loss of generality, let us simplify the reconstruction algorithm by setting $\theta_1=0$ and $\theta_2=\pi/2$ in \eqref{ex_form_recover}.

\begin{example}
    Let $\ell=100 \; {\rm ps}$. Supposing that  $x_c=(8,\,7,\,20)$ and $x_c=(8,\,7,\,30)$, we show the results of numerical reconstructions from noise-free measurements for different initial S-D pairs in Table \ref{tbl:ex_small_lt01}. The results of the measurements containing different noise levels are shown in Table \ref{tbl:ex_small_lt02}. 
\end{example}

\begin{table}[htp]
\centering
\caption{Reconstructions from noise-free measurements for different initial S-D pairs ($\ell=100 \; {\rm ps}$)}
\label{tbl:ex_small_lt01}
\begin{tabular}{cccc}
\hline
$x_c$& $\{x_d,\,x_s\}$ & $x_c^{inv}$  &  $RelErr$    \\ \hline
  \multirow{4}*{(8,\,7,\,20)} &\{(14,\,10,\,0),\,(6,\,10,\,0)\}& (8.58,\,7.40,\,20.18) & 3.32e-02 \\
                           & \{(8,\,5,\,0),\,(0,\,5,\,0)\}&(8.34,\,7.26,\,19.96)  & 1.92e-02 \\
                           &\{(5,\,0,\,0),\,(5,\,8,\,0)\}&(8.24,\,7.39,\,19.97)  & 2.01e-02 \\
                          &\{(16,\,15,\,0),\,(8,\,15,\,0)\}&(8.63,\,7.52,\,20.32)   & 3.88e-02  \\\hline
  \multirow{4}*{(8,\,7,\,30)}  &\{(14,\,10,\,0),\,(6,\,10,\,0)\}& (8.47,\,7.35,\,30.00) & 1.83e-02 \\
                           & \{(8,\,5,\,0),\,(0,\,5,\,0)\}&(8.33,\,7.26,\,29.87)  & 1.38e-02 \\
                           &\{(5,\,0,\,0),\,(5,\,8,\,0)\}&(8.24,\,7.35,\,29.88)  & 1.40e-02 \\
                          &\{(16,\,15,\,0),\,(8,\,15,\,0)\}&(8.50,\,7.43,\,30.09)   & 2.09e-02  \\\hline           
\end{tabular}
\end{table}

\begin{table}[htp]
\centering
\caption{Reconstructions from noisy measurements for different initial S-D pairs ($\ell=100 \; {\rm ps}$)}
\label{tbl:ex_small_lt02}
\begin{tabular}{cccc}
\hline
$x_c$& $\hat\delta$ & $x_c^{inv}$  &  $RelErr$    \\ \hline
  \multirow{3}*{(8,\,7,\,20)} &0.1\%& (8.62,\,7.42,\,20.20) & 3.43e-02 \\
                           & 1\%&(9.02,\,7.49,\,20.41)  & 5.32e-02 \\
                           &5\%&(10.73,\,7.84,\,21.23)  & 1.39e-01 \\\hline
\end{tabular}
\end{table}

Next, we turn to test the reconstruction algorithm for $\ell\gg 1$.  
\begin{example}
    Let $\ell=1000\; {\rm ps}$. Supposing that  $x_c=(8,\,7,\,20)$ and $x_c=(8,\,7,\,30)$, we show the results of numerical reconstructions from noise-free measurements for different initial S-D pairs in Table \ref{tbl:ex_large_lt01}. The results of the measurements containing different noise levels are shown in Table \ref{tbl:ex_large_lt02}. 
\end{example}

\begin{table}[htp]
\centering
\caption{Reconstructions from noise-free measurements for different initial S-D pairs ($\ell=1000 \; {\rm ps}$)}
\label{tbl:ex_large_lt01}
\begin{tabular}{cccc}
\hline
$x_c$& $\{x_d,\,x_s\}$ & $x_c^{inv}$  &  $RelErr$    \\ \hline
  \multirow{4}*{(8,\,7,\,20)} &\{(14,\,10,\,0),\,(6,\,10,\,0)\}& (8.76,\,7.60,\,20.28) & 4.47e-02 \\
                           & \{(8,\,5,\,0),\,(0,\,5,\,0)\}&(8.45,\,7.40,\,19.98)  & 2.65e-02 \\
                           &\{(5,\,0,\,0),\,(5,\,8,\,0)\}&(8.36,\,7.50,\,19.99)  & 2.73e-02 \\
                          &\{(16,\,15,\,0),\,(8,\,15,\,0)\}&(8.82,\,7.76,\,20.47)   & 5.34e-02  \\\hline
  \multirow{4}*{(8,\,7,\,30)}  &\{(14,\,10,\,0),\,(6,\,10,\,0)\}& (8.58,\,7.47,\,29.94) & 2.36e-02 \\
                           & \{(8,\,5,\,0),\,(0,\,5,\,0)\}&(8.42,\,7.36,\,29.77)  & 1.89e-02 \\
                           &\{(5,\,0,\,0),\,(5,\,8,\,0)\}&(8.34,\,7.45,\,29.78)  & 1.91e-02 \\
                          &\{(16,\,15,\,0),\,(8,\,15,\,0)\}&(8.62,\,7.55,\,30.05)   & 2.60e-02  \\\hline                      
\end{tabular}
\end{table}

From Table \ref{tbl:ex_large_lt01}, it is obvious that the algorithm is insensitive to the selection of the initial S-D pair and has a good reconstruction for a deeper unknown target. To show the robustness, we show the results to reconstruct $x_c=(10,\,10,\,30)$ by setting the initial S-D pair as $\{x_d,\,x_s\}=\{(14,\,10,\,0),\,(6,\,10,\,0)\}$ for different noise levels. 

\begin{table}[htp]
\centering
\caption{Reconstructions from noisy measurements for different initial S-D pairs ($\ell=1000\; {\rm ps}$)}
\label{tbl:ex_large_lt02}
\begin{tabular}{cccc}
\hline
$x_c$& $\hat\delta$ & $x_c^{inv}$  &  $RelErr$    \\ \hline
  \multirow{3}*{(8,\,7,\,30)} &0.1\%& (8.65,\,7.48,\,29.98) & 2.56e-02 \\
                           & 1\%&(9.28,\,7.62,\,30.29)  & 4.54e-02 \\
                           &5\%&(9.28,\,8.20,\,31.52)  & 1.39e-01 \\\hline
\end{tabular}
\end{table}

Compared with the relative error of the reconstruction from the noise-free measurement, the relative errors listed in Table \ref{tbl:ex_large_lt02} increase significantly as noise increases. Surprisingly, the reconstructed depth of the target is still acceptable for $\hat\delta=5\%$.

%%%%%%%%%%%%%%%%%%%%%%%%%%%%%%%%%%%%%%%%%%%%%%%%%%%%%%%%%%%%%
\section{Conclusions and future work}\label{sec_conc}
We proposed a direct and non-iterative inversion scheme to reconstruct the location of a point target using the measured peak time $t>0$ by recovering the relationship to the distance $\lambda>0$
\begin{equation}\label{Distan}
\lambda = \frac{|x_d-x_c|^2+|x_s-x_c|^2}{2vD},
\end{equation}
where $x_d$ is the detector point, $x_c$ is the unknown target, $x_s$ is the source point. The relationship between the peak time and the distance is known by certain nonlinear equations by dividing the cases that fluorescence lifetime is small in \cite{Chen2023} and large in \cite{Chen2025}. By asymptotic analysis in Theorems \ref{thm_lzero} and \ref{thm_linfty}, we are able to see the explicit relationship between the peak time and the distance $\lambda>0$. Theoretically also numerically, it is shown that the peak time has the same order to the distance when $\lambda \gg 1$ (See (c) in Figures \ref{fig_tpszero_diffpara}, \ref{fig_tps_diffpara} and \ref{fig_tps_linear_diffpara}). We are able to determine the distance $\lambda>0$ explicitly from the peak time which implies that the sphere for the target is identified, where the target lies on the sphere. By constructing a tetrahedron with edges determined by the radius of the identified sphere, we can find the location of the unknwon target $x_c$ as the vertex point of the tetrahedron by only using three S-D pairs  $\left\{\{x_s^{(n)},\,x_d^{(n)}\}\right\}_{n=1}^3$.

As a future work, we generalize the study on the relationship between the peak time and the distance $\lambda>0$ to the cases
where a target is not a point target, and the measurement surface $\partial\Omega$ is curved. 
\begin{comment}
The first task to start this study will be to have the Green function for the FDOT under the setup, including these cases. This is available by easily modifying the argument in one of our coauthors' papers \cite{Nakamura} giving the Green function for the interior transmission problem. The advantage of the mentioned argument is based on using the parabolic scaling which immediately gives the dominant part of the Green function. In relation to this, we note that the distance function is invariant under the parabolic scaling. Assuming the targets are well-separated convex domains and looking at the dominant part of the Green function, we speculate that we will find a similar situation as for the point target case in a neighborhood of the point that the mentioned sphere touches the target.
\end{comment}


\begin{thebibliography}{99}
\setlength{\itemsep}{0mm}

\bibitem{Chen2023} S. Chen, J. Eom, G. Nakamura, G. Nishimura, Approximate peak time and its application to time-domain fluorescence diffuse optical tomography, {\it Commun. Anal. Comput.}, 1 (2023) 379--406.

\bibitem{Chen2025} S. Chen, J. Eom, G. Nakamura, G. Nishimura, Approximate peak time to time-domain fluorescence diffuse optical tomography for nonzero fluorescence lifetime, 
submitted.

\bibitem{Eom2023a} J. Eom, M. Machida, G. Nakamura, G. Nishimura, C. Sun, Expressions of the peak time for time-domain boundary measurements of diffuse light, {\it J. Math. Phys.}, 64 (2023) 083504. %https://doi.org/10.1063/5.0081169. 

\bibitem{Eom2023b} J. Eom, G. Nakamura, G. Nishimura, C.~Sun, Local analysis for locating a single point target in time-domain fluorescence diffuse optical tomography, Differ. Integral Equ. 37 (2024) 27--58. %https://doi.org/10.57262/die037-0102-27.

\bibitem{Hebden1991} J. Hebden, R. Kruger, K. Wong, Time resolved imaging through a highly scattering medium, {\it Appl. Opt.}, 30 (1991) 788--794. %https://doi.org/10.1364/AO.30.000788
% review of the method of DOT

% DOT textbook recently the author published FDOT version
\bibitem{Jiang2011} H. Jiang, {\it Diffuse optical tomography: principles and applications}, CRC Press, Boca Raton, 2010.

\bibitem{Jiang2022} H. Jiang, {\it Fluorescence Molecular Tomography: Principles and Applications}, Springer, Cham, Switzerland, 2022.

\bibitem{Liu2022} J. Liu, M. Machida, G. Nakamura, G. Nishimura, C. Sun, On fluorescence imaging: the diffusion equation model and recovery of the absorption coefficient of fluorophores,
{\it Sci. China Math.}, 65 (2022) 1179--1198. %https://doi.org/10.1007/s11425-020-1731-y.

\bibitem{Ammari2020} Y. Liu, W. Ren, H. Ammari, Robust reconstruction of fluorescence molecular tomography with an optimized illumination pattern, {\it Inverse Probl. Imag.} 14 (2020) 535--568. %https://doi.org/10.3934/ipi.2020025.

% General textbook on the fluorescence in biomedical applications
\bibitem{Mycek2004} M. Mycek, B. Pogue, {\it Handbook of Biomedical Fluorescence}, Marcel Dekker, New York, 2003.

% One of most famous paper on FMT
\bibitem{Vasilis2002} V. Nitziachristos, C. Tung, C. Bremer, R. Weissleder, Fluorescence molecular tomography resolves protease activity in vivo, {\it Nat. Med.}, 8 (2002) 757--761. %https://doi.org/10.1038/nm729.

\end{thebibliography}
\end{document}